\documentclass[12pt]{article}
\usepackage{amsmath,amssymb,latexsym}
\usepackage{enumerate}
\usepackage{amsthm}
\usepackage{epsfig}

\newcommand\NN{\mathbb{N}}
\newcommand\RR{\mathbb{R}}
\newcommand\ZZ{\mathbb{Z}}
\newcommand\MM{\mathcal{M}}
\newcommand\BB{\mathcal{B}}

\newcommand\FF{\mathcal{F}}
\newcommand\GG{\mathcal{G}}
\newcommand\HH{\mathcal{H}}

\newcommand\eps{{\varepsilon}}
\DeclareMathOperator\sP{P}   %probability, treated as an operator
\newcommand{\rP}{\mathrm{P}} %probability, when spacing isn't a worry
   
\newcommand{\rE}{\mathrm{E}} 

\DeclareMathOperator\dist{dist}

\DeclareMathOperator\defin{def}
\DeclareMathOperator\ord{ord}
\renewcommand{\mod}{\mathrm{mod}\ }

\newtheorem{theorem}{Theorem}[section]
\newtheorem{definition}[theorem]{Definition}

\newtheorem{lemma}[theorem]{Lemma}
\newtheorem{remark}[theorem]{Remark}

\newtheorem{problem}{Problem}[section]

\newcommand{\address}{Address: Department of Mathematics, University of North Texas, 1155 Union Circle \#311430, Denton, TX 76203-5017, USA; E-mail: allaart@unt.edu, kiko@unt.edu}

\numberwithin{equation}{section}

\title{The Takagi function: a survey}
\author{Pieter C. Allaart and Kiko Kawamura \footnote{\address}}

\begin{document}

\maketitle

\section{Introduction}

More than a century has passed since Takagi \cite{Takagi} published his simple example of a continuous but nowhere differentiable function, yet Takagi's function -- as it is now commonly referred to despite repeated rediscovery by mathematicians in the West -- continues to inspire, fascinate and puzzle researchers as never before. For this reason, and also because we have noticed that many aspects of the Takagi function continue to be rediscovered with alarming frequency, we feel the time has come for a comprehensive review of the literature. Our goal is not only to give an overview of the history and known characteristics of the function, but also to discuss some of the fascinating applications it has found -- some quite recently! -- in such diverse areas of mathematics as number theory, combinatorics, and analysis. We also include a section on generalizations and variations of the Takagi function. In view of the overwhelming amount of literature, however, we have chosen to limit ourselves to functions based on the ``tent map". In particular, this paper shall not make more than a passing mention of the Weierstrass function and is not intended as a general overview of continuous nowhere-differentiable functions. We thank Prof. Paul Humke for encouraging us to write this survey, and for issuing periodic cheerful reminders.

\subsection{Early history}

\begin{figure}
\begin{center}
\epsfig{file=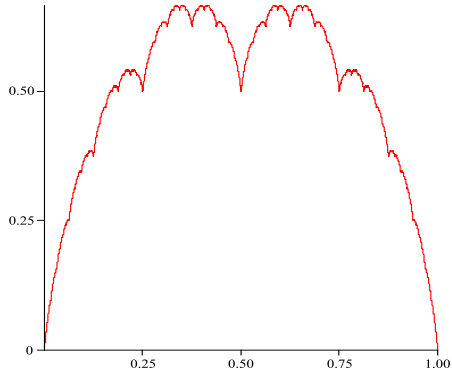, height=.25\textheight, width=.4\textwidth}
\caption{Graph of the Takagi function}
\label{fig:Takagi}
\end{center}
\end{figure}

Takagi's function is indeed simple: in modern notation, it is defined by
\begin{equation}
T(x)=\sum_{n=0}^\infty \frac{1}{2^n}\phi(2^n x),
\label{eq:Takagi-def}
\end{equation}
where $\phi(x)=\dist(x,\ZZ)$, the distance from $x$ to the nearest integer. The graph of $T$ is shown in Figure \ref{fig:Takagi}. Takagi himself expressed his function differently, and this is perhaps one reason (in combination with Japan's isolation at the beginning of the twentieth century) why it was largely overlooked in the West. Unlike for the more famous Weierstrass function, it is easy to show that $T$ has at no point a finite derivative; we include the short proof due to Billingsley \cite{Billingsley} in Section \ref{subsec:derivatives}. However, it does possess an infinite derivative at many points, and for this reason Knopp \cite{Knopp}, in his 1918 review of the rapidly growing body of ``strange" functions, did not consider it truely nowhere differentiable. Knopp outlined his own geometric method for producing functions which have no derivative, finite or infinite, at any point. The example most similar to the Takagi function is
$$f(x)=\sum_{n=0}^\infty a^n\phi(b^n x),$$
where $0<a<1$, $b$ is an integer and $ab>4$. Knopp's general construction also includes the Weierstrass function and Faber's example \cite{Faber} as special cases.

A variant of Takagi's function, using the base 10, was discovered in 1930 by Van der Waerden, who credits Dr.\,A.\,Heyting for sending in the proof of its nondifferentiability, in response to a problem in the publication {\em Wiskundige Opgaven} of the Dutch Mathematical Society. Three years later, Hildebrandt \cite{Hildebrandt} showed that one can more simply use the base 2, thus rediscovering Takagi's original example. An editorial note affixed to Hildebrandt's paper solicited answers to the ``interesting and probably not too difficult" question at which set of points $T(x)$ has an infinite derivative. Surprisingly, it would take 77 years for this natural question to be answered correctly, perhaps because the answer is not all that easy to guess; see Section \ref{subsec:derivatives} below.

In 1939 the Takagi function was rediscovered by Tambs-Lyche \cite{Tambs-Lyche}, who was also inspired by Van der Waerden's paper and motivated by a desire to give ``an example easy to understand for beginning students of analysis". Tambs-Lyche defined it by a formula different from both \eqref{eq:Takagi-def} and Takagi's original definition (see Section \ref{sec:expressions}), and stated without proof its equivalence to \eqref{eq:Takagi-def}. Tambs-Lyche was also the first to publish a graph of Takagi's function, hand-drawn but remarkably accurate.

Apparently unaware of Hildebrandt's and Tambs-Lyche's notes, de Rham \cite{deRham1} rediscovered Takagi's function once more in 1957. His main contribution, however, was to identify $T$ as a member of a more general class of functions which are solutions to a certain family of functional equations; in today's language, de Rham observed that the Takagi function is self-affine. His paper soon inspired Kahane \cite{Kahane} to determine the points of global and local extremum of $T$, and to modify the definition \eqref{eq:Takagi-def} in order to create functions with a prescribed modulus of continuity.

By the 1960s, the Takagi-van der Waerden function was sufficiently well known that it could be used as the key element in solutions to other problems, both in classical real analysis and in number theory. Lipinski \cite{Lipinski} used it in his elegant characterization of zero sets of continuous nowhere differentiable functions, after Schubert \cite{Schubert} had started the investigation. And Trollope \cite{Trollope} observed that the Takagi function was the missing piece of the puzzle in the binary digital sum problem; his proof was simplified and extended further by Delange \cite{Delange}. These applications are described in detail in Sections \ref{subsec:app-number-theory} and \ref{subsec:app-analysis}, respectively.

\subsection{The Takagi function comes home}

Interest in the Takagi-van der Waerden function spiked after 1980, with two more or less independent streams of publications. In the West, Billingsley \cite{Billingsley} drew new attention to the function with his short note in the {\em American Mathematical Monthly}, providing perhaps the most lucid proof of the function's nowhere differentiability. His argument was modified by Cater \cite{Cater} to show that $T$ does not even possess a finite one-sided derivative anywhere, and Shidfar and Sabetfakhri \cite{Shidfar} proved that $T$ is Lipschitz of every order $\alpha<1$. A sharper result was obtained by Mauldin and Williams \cite{Mauldin-Williams}, who investigated a much larger class of functions defined by infinite series and showed that the Takagi function is ``convex Lipschitz" of order $h\log(1/h)$. Anderson and Pitt \cite{Anderson} slightly improved on this by showing that
\begin{equation*}
T(x+h)-T(x)=O(h\log(1/|h|), \quad\mbox{as $h\to 0$},
\end{equation*}
and this estimate is the best possible. As a result, the Hausdorff dimension of the graph of $T$ is one.

Meanwhile, the Takagi function had become popular in the country of its birth, due to the influential paper by Hata and Yamaguti \cite{Hata-Yamaguti}. Besides finally restoring credit to its original inventor, these authors did much to elevate Takagi's function beyond the realm of recreational mathematics, by pointing out its connection with chaotic dynamical systems and proving a beautiful relationship between $T$ and Lebesgue's singular function (also called Salem's function or the Riesz-Nagy function). Hata and Yamaguti also replaced the factor $1/2^n$ in \eqref{eq:Takagi-def} by an arbitrary constant $c_n$, calling their new family of functions the Takagi class. K\^ono \cite{Kono} characterized completely the differentiability properties of members of the Takagi class -- there are three qualitatively different cases -- and proved several other results about these functions, most of them of a probabilistic nature. Gamkrelidze \cite{Gamkrelidze} later applied K\^ono's methods to obtain a Central Limit Theorem-type result for the small-scale oscillations of $T$. In the same year as Hata and Yamaguti's paper, Baba \cite{Baba} calculated the maxima of the general Takagi-van der Waerden function (with arbitrary base $r\geq 2$), after Martynov \cite{Martynov1} had rediscovered Kahane's result about the maximum of $T$. Tsujii \cite{Tsujii} constructed a Takagi-like function of two variables, and Yamaguchi et al. \cite{YamTanMi} viewed the graph of $T$ as the invariant repeller of a dynamical system.

\subsection{Recent work}

In the last two decades, the literature on the Takagi function and related topics seems to have grown exponentially. Papers from this period can be loosely classified into three categories: papers about the Takagi function itself, papers dealing primarily with applications, and papers discussing various generalizations and variations. Some papers fit more than one category. It is impossible to describe each individual contribution in this introduction. We limit ourselves here to succinct groupings of papers by topic, referring to later sections for the details.

\bigskip
1. {\em Papers about the Takagi function itself.} These can be further divided as follows. Kairies et al. \cite{KaiDarFra} and Kairies \cite{Kairies} characterize $T$ by its functional equations. Other papers, such as Brown and Kozlowski \cite{Brown}, Abbott et al. \cite{Abbott} and Watanabe \cite{Watanabe}, focus on various local continuity properties. The infinite derivatives of $T$ are dealt with in Kr\"uppel \cite{Kruppel1,Kruppel3} and Allaart and Kawamura \cite{AK2}. There is also a sustained effort ongoing to understand the complicated level set structure of $T$; see Buczolich \cite{Buczolich2}, Maddock \cite{Maddock}, Lagarias and Maddock \cite{LagMad1,LagMad2}, Allaart \cite{Allaart4,Allaart5}, and de Amo et al. \cite{ABDF}. A richly illustrated expository article by Martynov \cite{Martynov2} gives step-by-step explanations of the main characteristics of $T$, aimed at undergraduate students.

\bigskip
2. {\em Papers concerned with applications.} A number of authors have extended Trollope's result about binary digital sums in various directions. The papers most closely related to the Takagi function are Okada et al. \cite{Okada}, Kobayashi \cite{Kobayashi} and Kr\"uppel \cite{Kruppel2}. Recently, H\'azy and P\'ales \cite{Hazy-Pales} and Boros \cite{Boros} found Takagi's function to be the extremal case in the theory of approximately midconvex functions. Their work was elaborated on by Tabor and Tabor \cite{Tabor1,Tabor2} and Mak\'o and P\'ales \cite{Mako-Pales}, and this last paper includes many further references. Allaart \cite{Allaart3} reduces the crucial inequality in the above papers to a simple inequality for binary digital sums, thus linking the two applications.
Takagi's function also arises naturally as the limit in certain counting problems in graph theory; see Frankl et al. \cite{Frankl}, Knuth \cite{Knuth} and Guu \cite{Guu}. It even has been used in an equivalent statement of the Riemann hypothesis; see Balasubramanian et al. \cite{Balasub}.

\bigskip
3. {\em Papers about generalizations and variations.} There are literally hundreds of papers about generalizations of the Takagi function. Many replace the tent map $\phi$ with a more general bounded ``base" function; some also introduce random phase shifts. We will not discuss such functions here, but limit ourselves to  generalizations based on the tent map. The most direct extension, a subcollection of the Takagi class, is the family of functions $f_\alpha(x)=\sum_{n=0}^\infty\alpha^n\phi(2^n x)$ where $0<\alpha<1$. They were studied for $\alpha>1/2$ by Ledrappier \cite{Ledrappier}, who computed their Hausdorff dimension, and for $\alpha<1/2$ by Tabor and Tabor \cite{Tabor1,Tabor2} and Allaart \cite{Allaart3}. Other extensions allow the ``tents" at each stage of the construction to be flipped up or down individually. This way one obtains for instance the Gray Takagi function of Kobayashi \cite{Kobayashi} or the function $T^3$ of Kawamura \cite{Kawamura1}. Anderson and Pitt \cite{Anderson}, Abbott et al. \cite{Abbott} and Allaart \cite{Allaart2} investigate general properties of this larger class of functions. Sekiguchi and Shiota \cite{Sekiguchi-Shiota}, generalizing the work of Hata and Yamaguti, obtained another family of continuous functions, which were examined more closely by Allaart and Kawamura \cite{AK1}. A version of the Takagi function with random signs is studied in Allaart \cite{Allaart1}, and Kawamura \cite{Kawamura2} considers the composition of $T$ with a singular function. Finally, Sumi \cite{Sumi} introduces a complex version of the Takagi function in connection with random dynamics in the complex plane.

\subsection{Organization of this paper}

This survey is organized as follows. Section \ref{sec:analytic} focuses on analytic aspects of the Takagi function. We give Billingsley's proof of the nowhere-differentiability of $T$ and characterize the set of points where $T$ has an infinite derivative. In Section \ref{subsec:continuity}, we treat the H\"older continuity of $T$ and explain the work of Abbott et al. \cite{Abbott} regarding slow points. This is followed by a more detailed examination of the modulus of continuity of $T$.

Section \ref{sec:graph} deals with the graph of $T$. We first discuss the global and local extrema of $T$. Then we point out the partial self-similarity of the graph and illustrate how to use this to prove a specific theorem, namely that the graph of $T$ has $\sigma$-finite linear measure.

In Section \ref{sec:expressions} we give a number of different expressions for $T(x)$, show how these can be derived from one another, and explain how they have been used to prove various aspects of the Takagi function.

Section \ref{sec:FE} gives functional equations for $T$, presents $T$ as the unique bounded solution of a system of infinitely many difference equations, and discusses the connection of $T$ with Lebesgue's singular function.

Section \ref{sec:level-sets} is devoted to the level sets of $T$. This area of research is currently very active: Nearly all the results in this section were found in the last five years or so. We outline a proof, based on the partial self-similarity ideas of Section \ref{sec:graph}, of the fact that almost all level sets of $T$ are finite, and give an overview of the other known facts about the level sets. The section ends with a list of open problems.

Section \ref{sec:generalizations} gives an overview of some of the generalizations of the Takagi function and other related functions. This includes the general Takagi-van der Waerden functions, the Takagi class, and the Zygmund spaces $\Lambda_d^*$, $\lambda_d^*$ and $\Lambda_{d,1}^*$, all of which are in some sense fairly direct extensions of the Takagi function. This section too concludes with a list of open problems.

Section \ref{sec:applications} deals with applications, and is divided into four parts. Subsection \ref{subsec:app-number-theory} presents Trollope's formula for the sum of binary digits of the first $N$ positive integers and discusses several related results. In Subsection \ref{subsec:app-combinatorics} we treat applications of the Takagi function to the problem of finding the minimum shadow size in uniform hypergraphs and to the edge-discrete isoperimetric problem on the $n$-cube. Subsection \ref{subsec:app-analysis} deals with applications in classical real analysis and consists of two parts: one on the use of $T$ in Lipinski's characterization of zero sets of continuous nowhere differentiable functions, and one on the role of $T$ and its generalizations in approximate convexity problems. Finally, Subsection \ref{subsec:app-Riemann} explains the connection between Takagi's function and the Riemann hypothesis.

\bigskip
We have not attempted to give equal coverage to all the players in this arena. The things we have chosen to emphasize reflect our interests and expertise, not the importance or quality of the cited works.

While we were preparing this article, we learned that Jeffrey Lagarias \cite{Lagarias} was working on a survey paper of his own. The two surveys evolved for the most part independently, and while there is inevitably a considerable degree of overlap, the two surveys emphasize different things. For example, we treat in detail the differentiability aspects and fine structure of the graph of $T$, and discuss various generalizations and applications in considerable detail. (Hence the length of the last two sections of this paper.) Lagarias, on the other hand, focuses on connections of the Takagi function with several areas of analysis, including wavelets, complex power series, and dynamical systems. In view of this, we feel that our survey and that of Lagarias complement each other quite well.

\subsection{Frequently used notation} \label{subsec:notation}

We collect here some notation that will be used regularly throughout this paper. First, for definiteness, we let $\NN$ denote the set of natural numbers, and $\ZZ_+$ the set of nonnegative integers. Most important is the binary expansion of a point $x\in[0,1)$, which we denote by 
\begin{equation*}
x=\sum_{n=1}^\infty \frac{\eps_n}{2^n}=0.\eps_1\eps_2\dots\eps_n\dots, \qquad \eps_n\in\{0,1\}.
\end{equation*}
For dyadic rational $x$ (i.e. $x$ of the form $x=k/2^m$ with $k\in\ZZ_+$ and $m\in\NN$) we choose the representation ending in all zeros.
When necessary, to avoid confusion, we write $\eps_n(x)$ instead of $\eps_n$.

Let $I_n:=I_n(x)$ be the number of ones, and $O_n:=O_n(x)$ the number of zeros, among the first $n$ binary digits of $x$, and let $D_n:=D_n(x):=O_n(x)-I_n(x)$. Thus, we have
\begin{equation*}
I_n=\sum_{k=1}^n \eps_n, \qquad O_n=n-I_n,
\end{equation*}
and
\begin{equation*}
D_n=\sum_{k=1}^n(1-2\eps_k)=\sum_{k=1}^n (-1)^{\eps_k}.
\end{equation*}
If the limit
\begin{equation}
d_1(x):=\lim_{n\to\infty}\frac{1}{n}\sum_{k=1}^{n}\eps_k,
\label{eq:density1}
\end{equation}
exists, we call $d_1(x)$ the {\em density} (or long-run frequency) of the digit ``$1$" in the binary expansion of $x$. In that case, the number
\begin{equation*}
d_0(x):=1-d_1(x)
\end{equation*}
is the density of the digit ``$0$". 

The orthogonal projections onto the $x$- and $y$-axes will be denoted by $\pi_X$ and $\pi_Y$, respectively.

By $\HH^\alpha$ we will denote $\alpha$-dimensional Hausdorff measure, and by $\dim_H A$, the Hausdorff dimension of a set $A$. For a function $f$, $\dim_H(f)$ denotes the Hausdorff dimension of the graph of $f$.

\section{Analytic properties} \label{sec:analytic}

In this section we focus on analytic aspects of the Takagi function, including infinite derivatives, H\"older continuity and slow points. We begin with a short proof of the function's nowhere-differentiability.

\subsection{Derivatives, or lack thereof} \label{subsec:derivatives}

Takagi himself gave a proof of the fact that $T$ has nowhere a finite derivative \cite{Takagi}, as did Hildebrandt \cite{Hildebrandt} and de Rham \cite{deRham1}. Van der Waerden's simple proof for the base 10 case, however, does not immediately transfer to the case of base 2. While all published proofs of non-differentiability follow more or less the same logic, the one by Billingsley \cite{Billingsley} is arguably the most natural, and that is the one we present here.

\begin{theorem}[Takagi]
The Takagi function $T$ does not possess a finite derivative at any point.
\end{theorem}

\begin{proof}(Billingsley)
Put $\phi_k(x)=2^{-k}\phi(2^k x)$ for $k=0,1,\dots$. Fix a point $x$, and, for each $n\in\NN$, let $u_n$ and $v_n$ be dyadic rationals of order $n$ with $v_n-u_n=2^{-n}$ and $u_n\leq x<v_n$. Then
\begin{equation*}
\frac{T(v_n)-T(u_n)}{v_n-u_n}=\sum_{k=0}^{n-1}\frac{\phi_k(v_n)-\phi_k(u_n)}{v_n-u_n},
\end{equation*}
since $\phi_k(u_n)=\phi_k(v_n)=0$ for all $k\geq n$. But for $k<n$, $\phi_k$ is linear on $[u_n,v_n]$ with slope $\phi_k^+(x)$, the right-hand derivative of $\phi_k$ at $x$. Thus,
\begin{equation*}
\frac{T(v_n)-T(u_n)}{v_n-u_n}=\sum_{k=0}^{n-1}\phi_k^+(x).
\end{equation*}
Since $\phi_k^+(x)=\pm 1$ for each $k$, this last sum cannot converge to a finite limit. Hence, $T$ does not have a finite derivative at $x$.
\end{proof}

Billingsley's argument was modified by Cater \cite{Cater} to show that $T$ does not have a finite {\em one-sided} derivative anywhere. The above proof makes it plausible, however, that there exist points with $T'(x)=\pm\infty$. An Editor's note affixed to Hildebrandt's paper asked readers to characterize the set of such points. The call was answered three years later by Begle and Ayres \cite{Begle}, who claimed that $T'(x)=\infty$ if $D_n(x)\to\infty$, and $T'(x)=-\infty$ if $D_n(x)\to-\infty$. This is certainly believable at first sight: If we agree that for dyadic rational points we choose the binary expansion ending in all zeros, then the last equation in the above proof can be written as
\begin{equation}
\frac{T(v_n)-T(u_n)}{v_n-u_n}=D_n(x).
\label{eq:dyadic-secant-slopes}
\end{equation}
Convergence of the above slopes to $\pm\infty$ is necessary in order that $T'(x)=\pm\infty$, but it is of course, {\em a priori}, not sufficient. (In fact, there are examples of nowhere differentiable functions for which the dyadic derivative exists almost everywhere \cite[Example 3.3]{Anderson}.) Begle and Ayres assumed that for fixed $n$, the slope $D_n(x)$ cannot jump by more than $\pm 2$ as one moves from one dyadic interval into the next. But this is already false for $n=4$, as $D_4(x)=-2$ for $7/16\leq x<1/2$, and $D_4(x)=2$ for $1/2\leq x<9/16$.

The paper by Begle and Ayres appears to have been forgotten soon after its publication, as was Hildebrandt's note. In any case, there is no evidence in the literature that the mistake was ever noticed -- until a few years ago, that is. We learned of Begle and Ayres' work from an historical survey by Prof. H. Okamoto, written in Japanese. Knowing that Prof. M. Kr\"uppel \cite{Kruppel1} had recently written about the improper derivatives of the Takagi function, we sent him a courtesy notification. Kr\"uppel's stunning reply was that, while he did not know about Begle and Ayres, the result could not possibly be true, as his own paper contained a counterexample! (Curiously, Lemma 7.4 of Anderson and Pitt \cite{Anderson} implies the same incorrect statement. In their case, however, the culprit appears to be a typographical error.)

We present Kr\"uppel's example here in somewhat simplified form. Let $x=\sum_{n=1}^\infty 2^{-a_n}$, where $a_n=4^n$. Then certainly $D_n(x)\to\infty$. A well-known formula for $T(x)$ at dyadic rational points is
\begin{equation}
T\left(\frac{k}{2^m}\right)=\frac{1}{2^m}\sum_{j=0}^{k-1}(m-2s_j),
\label{eq:Takagi-dyadic}
\end{equation}
where $s_j$ is the number of ones in the binary representation of the integer $j$. (There are several ways to derive this formula; see Section \ref{sec:expressions}.) For given $m$, let $k$ be the integer such that $k/2^m<x<(k+1)/2^m$. Then
the secant slopes over the dyadic intervals $[k/2^m,(k+1)/2^m]$ containing $x$ indeed tend to $+\infty$ in view of \eqref{eq:dyadic-secant-slopes}. However, if we put $m=a_{n+1}-1$, then $s_k=n$, $s_{k-1}=n+a_{n+1}-a_n-2$, and $s_{k-2}=n+a_{n+1}-a_n-3$. Thus, \eqref{eq:Takagi-dyadic} yields
\begin{align*}
2^m\left[T\left(\frac{k+1}{2^m}\right)-T\left(\frac{k-2}{2^m}\right)\right]
&=3m-2s_k-2s_{k-1}-2s_{k-2}\\
&=4a_n-a_{n+1}-6n+7\\
&\to-\infty,
\end{align*}
as $n\to\infty$. Since the intervals $[(k-2)/2^m,(k+1)/2^m]$ also contain $x$, it follows that $T$ cannot have an infinite derivative at $x$. 

Intrigued by these developments, the present authors and Prof. Kr\"uppel independently set out to find the correct answer. The result:

\begin{theorem}[Allaart and Kawamura, Kr\"uppel] \label{thm:infinte-derivatives}
Let $x\in(0,1)$ be non-dyadic, and write
\begin{equation*}
x=\sum_{n=1}^\infty 2^{-a_n},
%\label{eq:expansions}
\end{equation*}
where $\{a_n\}$ is a strictly increasing sequence of positive integers. Then
$T'(x)=\infty$ if and only if
\begin{equation}
a_{n+1}-2a_n+2n-\log_2(a_{n+1}-a_n)\to-\infty.
\label{eq:NS-condition+}
\end{equation}
\end{theorem}

By the symmetry of the Takagi function, $T'(x)=-\infty$ if and only if $T'(1-x)=\infty$. It is easy to see from the definition \eqref{eq:Takagi-def} that if $x$ is a dyadic rational, then $T'_+(x)=+\infty$ and $T'_-(x)=-\infty$, where $T'_+$ and $T'_-$ denote the right- and left-hand derivatives of $T$, respectively. Combined with these facts, Theorem \ref{thm:infinte-derivatives} gives a complete characterization of the infinite derivatives of $T$.

In fact, the condition \eqref{eq:NS-condition+} is necessary in order that $T'_-(x)=+\infty$. For $T'_+(x)=+\infty$, it is sufficient that $a_n-2n\to\infty$, and this can be seen to be equivalent to the Begle and Ayres condition that $D_n(x)\to\infty$.
Allaart and Kawamura \cite{AK2} give several examples illustrating the condition \eqref{eq:NS-condition+}. For instance, the condition holds for $a_n=3n$; for any increasing polynomial of degree $2$ or higher; and for any exponential sequence $a_n=\lfloor\alpha^n\rfloor$ with $1<\alpha<2$. On the other hand, it fails whenever $\limsup_{n\to\infty}a_{n+1}/a_n>2$. The logarithmic term in \eqref{eq:NS-condition+} is sometimes a difference maker: The sequence $a_n=2^n$ does not satisfy \eqref{eq:NS-condition+}; neither does $a_n=2^n+n$. But $a_n=2^n+(1+\eps)n$ satisfies \eqref{eq:NS-condition+} for any $\eps>0$. 

Theorem \ref{thm:infinte-derivatives} implies that, if the density $d_1(x)$ exists and lies strictly between $0$ and $1/2$, then $T'(x)=\infty$. By symmetry, $T'(x)=-\infty$ if $1/2<d_1(x)<1$. As a result, the sets $\{x:T'(x)=\infty\}$ and $\{x:T'(x)=-\infty\}$ have Hausdorff dimension 1.

\subsection{Continuity properties} \label{subsec:continuity}

Since $T$ is nowhere differentiable, it is certainly not Lipschitz. However, Shidfar and Sabetfakhri \cite{Shidfar} showed that $T$ is H\"older continuous of any order $\alpha<1$. That is, for each $0<\alpha<1$, there is a constant $C_\alpha$ such that
\begin{equation*}
|T(x)-T(y)|\leq C_\alpha|x-y|^\alpha, 
\end{equation*}
for all $x$ and $y$ in $[0,1]$. This result prompted an interesting question. A theorem of Marcinkiewicz says that for every Lipschitz function $f$ on $[0,1]$ there is a $C^1$ function $g$ which agrees with $f$ outside a set of arbitrarily small measure, and Brown and Koslowski \cite{Brown} wondered if the Lipschitz requirement in this theorem can be replaced with the weaker condition that $f$ be H\"older continuous of any order $\alpha<1$. They show that this is not so: the Takagi function provides a counterexample, since for any set $M$ of positive Lebesgue measure, the set of difference quotients
\begin{equation*}
\left\{\frac{T(x)-T(y)}{x-y}: x\in M, y\in M \ \mbox{and}\ x\neq y\right\}
\end{equation*}
is unbounded.

While the Takagi function is not Lipschitz, it does satisfy a local Lipschitz condition at each ``slow point". Abbott et al. \cite{Abbott} call a point $x$ in $[0,1]$ a {\em slow point} with constant $K$ ($K$ a positive integer) if $|D_n(x)|\leq K$ for every $n$. They show that there is a uniform bound $P=P(K)>0$ such that for each slow point $x$ with constant $K$ and for each $y\in[0,1]$, $|T(y)-T(x)|\leq P|y-x|$. They also manage to compute the Hausdorff dimension of the set of slow points with constant $K$; it is $1+\log_2 r$, where $r=\cos(\pi/(2(K+1)))$. The paper by Abbott et al. also contains results for more general functions based on the tent map; we will return to it in Section \ref{sec:generalizations}.

The result of Shidfar and Sabetfakhri was sharpened by Anderson and Pitt \cite{Anderson}, who showed that not only $T$, but every function $f$ in the so-called Zygmund space $\Lambda_d^*$ is Lipschitz of order $\theta(y)=y\log(1/y)$; that is to say, there is a constant $M$ such that, for all $x$ and $y$ with $y>0$ sufficiently small,
\begin{equation*}
|f(x+y)-f(x)|\leq My\log(1/y).
\end{equation*}
From this, one can deduce that the graph of $T$ has Hausdorff dimension 1, a result first obtained by Mauldin and Williams \cite{Mauldin-Williams} on which we will elaborate in the next section.

More precise estimates on the oscillations of $T$ were obtained by K\^ono \cite{Kono}. He describes both the ``worst-case" behavior and the ``typical" size (in the Lebesgue sense) of the oscillations. Let
\begin{equation*}
\sigma_u(h)=\log_2(1/h) \qquad\mbox{and} \qquad \sigma_l(h)=\sqrt{\log_2(1/h)}, \qquad h>0.
\end{equation*}

\begin{theorem}[K\^ono 1987] \label{thm:uniform-modulus}
The oscillations of the Takagi function satisfy
\begin{equation*}
\limsup_{|x-y|\to 0}\frac{T(x)-T(y)}{(x-y)\sigma_u(|x-y|)}=1 =-\liminf_{|x-y|\to 0}\frac{T(x)-T(y)}{(x-y)\sigma_u(|x-y|)}.
\end{equation*}
\end{theorem}

The extremal case of the above theorem is rare -- at most points $x$ the oscillations are of a smaller order.

\begin{theorem}[K\^ono 1987] \label{thm:local-modulus}
For almost every $x\in[0,1]$, we have
\begin{equation*}
\limsup_{h\to 0} \frac{T(x+h)-T(x)}{h\sigma_l(|h|)\sqrt{2\log\log\sigma_l(|h|)}}=1 =-\liminf_{h\to 0} \frac{T(x+h)-T(x)}{h\sigma_l(|h|)\sqrt{2\log\log\sigma_l(|h|)}}.
\end{equation*}
\end{theorem}

K\^ono proves the last theorem by developing $T(x)$ in terms of Rademacher functions and applying the law of the iterated logarithm.
Note that for {\em fixed} $x\in[0,1]$, Theorem \ref{thm:uniform-modulus} implies
\begin{equation}
-1\leq\liminf_{h\to 0} \frac{T(x+h)-T(x)}{h\log_2(1/|h|)}\leq \limsup_{h\to 0} \frac{T(x+h)-T(x)}{h\log_2(1/|h|)}\leq 1.
\label{eq:max-oscillations}
\end{equation}
Within these bounds, various kinds of behavior are possible. Kr\"uppel \cite{Kruppel1} shows that if $x$ is dyadic rational, then
\begin{equation*}
\lim_{h \to 0} \frac{T(x+h)-T(x)}{|h| \log_2(1/|h|)}=1.
\end{equation*}
Allaart and Kawamura \cite{AK2} characterize for which points $x$ the limit
\begin{equation}
\lim_{h \to 0} \frac{T(x+h)-T(x)}{h \log_2(1/|h|)}
\label{eq:two-sided-limit}
\end{equation}
exists. This requires the following definition.

\begin{definition}
{\rm
Let $x\in[0,1]$ be non-dyadic, and let $\{a_n\}$ and $\{b_n\}$ be the (unique) strictly increasing sequences of positive integers such that
\begin{equation*}
x=\sum_{n=1}^\infty 2^{-a_n}, \qquad 1-x=\sum_{n=1}^\infty 2^{-b_n}.
%\label{eq:expansions}
\end{equation*}
We say $x$ is {\em density-regular} if $d_1(x)$ exists and one of the following holds:\vspace{0mm}
\begin{enumerate}[(a)]\setlength{\itemsep}{-1mm}
\item $0<d_1(x)<1$; or
\item $d_1(x)=0$ and $a_{n+1}/a_n\to 1$; or
\item $d_1(x)=1$ and $b_{n+1}/b_n\to 1$.
\end{enumerate}
}
\end{definition}

\begin{theorem}[Allaart and Kawamura 2010]
Let $x$ be non-dyadic. The limit in \eqref{eq:two-sided-limit} exists if and only if $x$ is density-regular, in which case the limit is equal to $d_0(x)-d_1(x)$.
\end{theorem}

One can also consider the probability distribution of $T(x+h)-T(x)$ for small $h$ when $x$ is chosen at random from $[0,1]$. Gamkrelidze \cite{Gamkrelidze} adapts K\^ono's approach to give the following Central Limit Theorem-type result:
\begin{equation*}
\lim_{h\downarrow 0} \lambda\left(\left\{x: \frac{T(x+h)-T(x)}{h\sqrt{\log_2(1/h)}}\leq y\right\}\right)=\frac{1}{\sqrt{2\pi}}\int_{-\infty}^y e^{-t^2/2}dt,
\end{equation*}
where $\lambda$ denotes Lebesgue measure.

\section{Graphical properties} \label{sec:graph}

Figure \ref{fig:Takagi} shows the graph of the Takagi function, restricted to the interval $[0,1]$. A number of features quickly jump out. The graph is symmetric about the line $x=1/2$, and it has cusps and local minima at the dyadic rational points. An important aspect of the graph is that it has two $1/4$-scale copies of itself at its top; this is due to the self-affine nature of $T$, and leads quickly to the fact, observed by many an author, that the absolute maximum of $T$ is attained at uncountably many points. More specifically, we have

\begin{theorem}[Kahane 1959] \label{thm:maxima}
The maximum value of $T$ is $2/3$. The set $\mathcal{M}$ of points where $T$ attains the maximum value is a perfect set of Hausdorff dimension $1/2$, and consists of all the points $x$ with binary expansion satisfying $\eps_{2n-1}+\eps_{2n}=1$ for each $n$.
\end{theorem}

\begin{proof}
The simplest way to see this is to rewrite \eqref{eq:Takagi-def} as
\begin{equation*}
T(x)=\sum_{n=0}^\infty \frac{1}{4^n}\phi_1(4^n x),
\end{equation*}
where $\phi_1$ is the ``table-top" function $\phi_1(x)=\phi(x)+(1/2)\phi(2x)$. Let
\begin{equation}
T_n(x):=\sum_{k=0}^{n-1}\frac{1}{2^k}\phi(2^k x),
\label{eq:partial-Takagi}
\end{equation}
and note that $T_{2n}(x)=\sum_{k=0}^{n-1}4^{-k}\phi_1(4^k x)$, for $n\in\NN$. Figure~\ref{fig:table-top} shows the graphs of $T_2$ and $T_4$. One sees by induction that the maximum value of $T_{2n}$ is
\begin{equation*}
\frac12+\frac12\cdot\frac14+\dots+\frac12\left(\frac14\right)^{n-1},
\end{equation*}
and hence,
\begin{equation*}
M:=\max\{T(x): x\in[0,1]\} =\sum_{k=0}^\infty\frac12\left(\frac14\right)^k=\frac23.
\end{equation*}
If $x\in[0,1]$, then $T(x)$ achieves this maximum value of $2/3$ if and only if $x$ lies in the middle half of each quarternary interval to which it belongs -- in other words, if the quarternary expansion of $x$ contains only 1's and 2's. In terms of the binary expansion $x=\sum_{n=0}^\infty \eps_n/2^n$ of $x$, it means that $\eps_{2n-1}+\eps_{2n}=1$ for each $n$. Thus, the set $\MM:=\{x\in[0,1]: T(x)=M\}$ is a Cantor-like set constructed by removing at each step the two outside fourths of each remaining quarternary interval. As a result, $\dim_H \MM=\log 2/\log 4=1/2$.
\end{proof}

(Kahane did not show that the dimension of $\mathcal{M}$ is $1/2$, but he easily could have: the technique for calculating the dimensions of generalized Cantor sets was by 1959 well established.)

\begin{figure}
\begin{center}
\epsfig{file=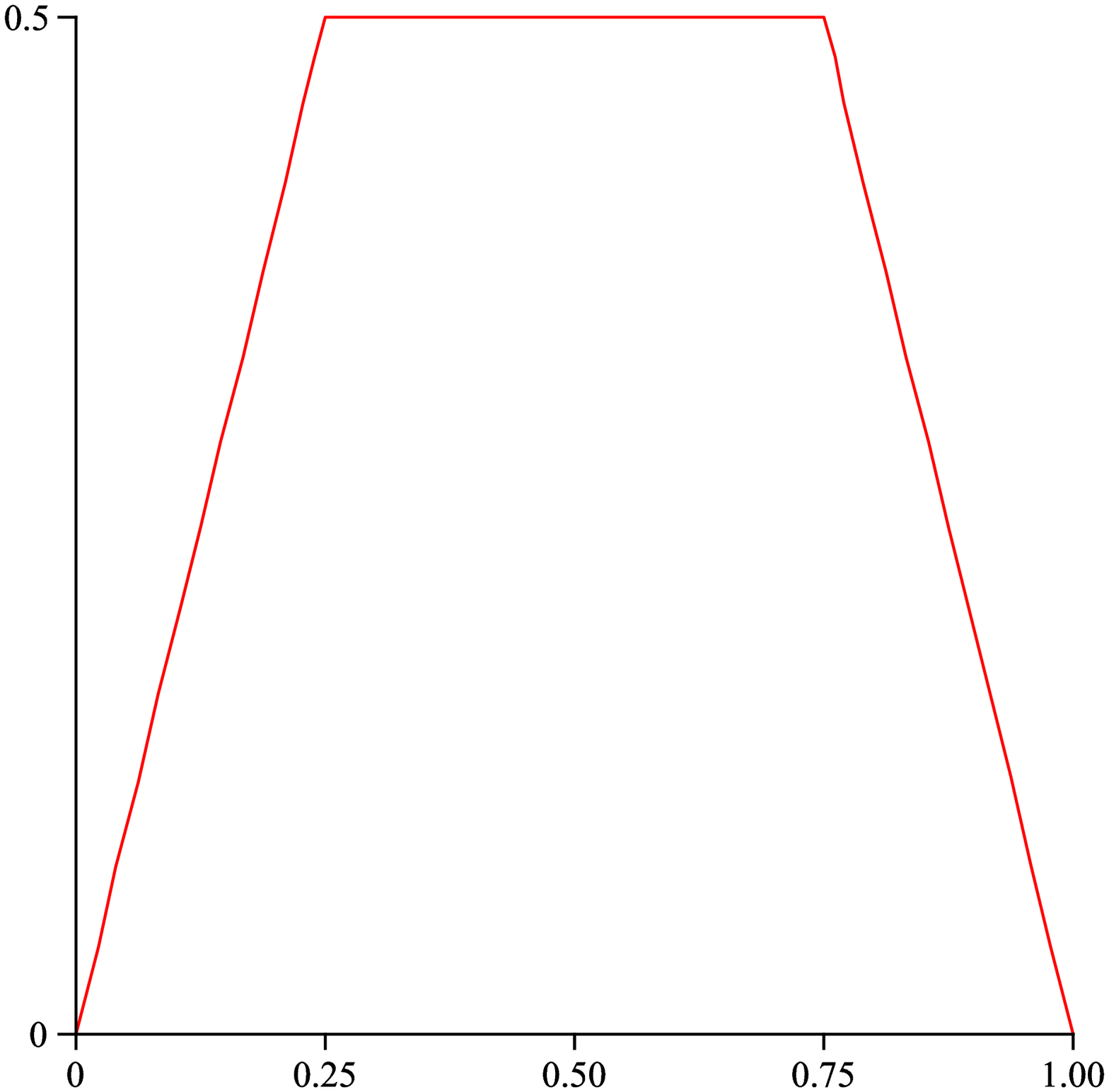, height=.2\textheight, width=.4\textwidth}
\qquad
\epsfig{file=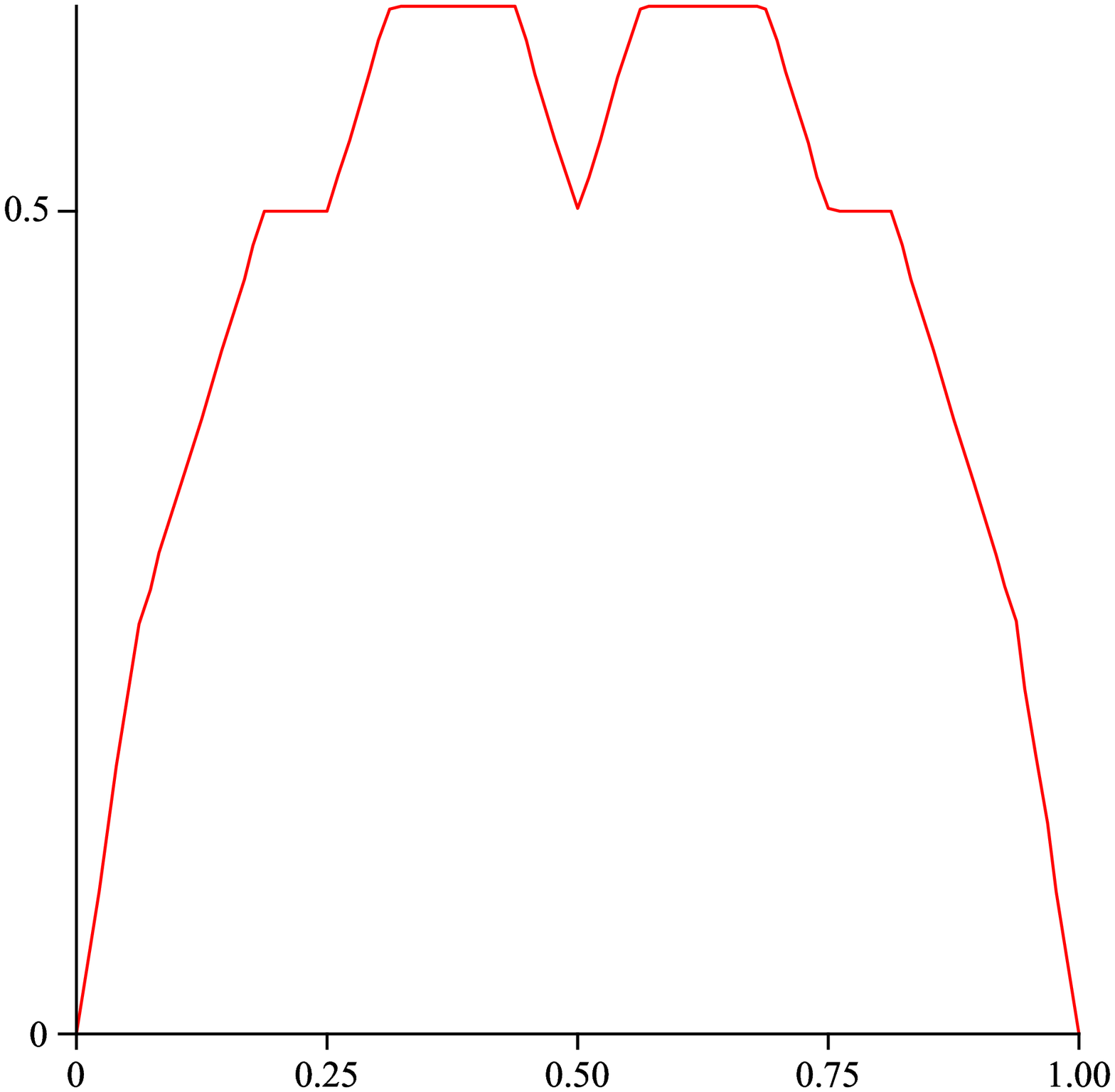, height=.25\textheight, width=.4\textwidth}
\caption{The functions $T_2(x)=\phi_1(x)$ (left) and $T_4(x)=\phi_1(x)+(1/4)\phi_1(4x)$ (right)}
\label{fig:table-top}
\end{center}
\end{figure}

\subsection{Humps and local extrema} \label{subsec:humps}

The appearance of smaller-scale similar copies of the graph of $T$ is not limited to the central part of the graph; it happens everywhere. We introduce two definitions and a lemma to make this precise. Let
\begin{equation*}
\mathcal{G}_T:=\{(x,T(x)): 0\leq x\leq 1\}
\end{equation*}
denote the graph of $T$ over the unit interval $[0,1]$.
The term `balanced' in the following definition is taken from Lagarias and Maddock \cite{LagMad1}.

\begin{definition}
{\rm
A dyadic rational of the form $x=0.\eps_1\eps_2\dots\eps_{2m}$ is called {\em balanced} if $D_{2m}(x)=0$. If there are exactly $n$ indices $1\leq j\leq 2m$ such that $D_j(x)=0$, we say $x$ is a balanced dyadic rational of {\em generation} $n$. By convention, we consider $x=0$ to be a balanced dyadic rational of generation $0$. The set of all balanced dyadic rationals is denoted by $\BB$. For each $n\in\ZZ_+$, the set of balanced dyadic rationals of generation $n$ is denoted by $\BB_n$. Thus, $\BB=\bigcup_{n=0}^\infty \BB_n$.
}
\end{definition}

\begin{lemma} \label{lem:similar-copies}
Let $m\in\NN$, and let $x_0=k/2^{2m}=0.\eps_1\eps_2\dots\eps_{2m}$ be a balanced dyadic rational. Then for $x\in[k/2^{2m},(k+1)/2^{2m}]$ we have
\begin{equation*}
T(x)=T(x_0)+\frac{1}{2^{2m}}T\left(2^{2m}(x-x_0)\right).
\end{equation*}
In other words, the part of the graph of $T$ above the interval $[k/2^{2m},(k+1)/2^{2m}]$ is a similar copy of the full graph $\mathcal{G}_T$, reduced by a factor $1/2^{2m}$ and shifted up by $T(x_0)$.
\end{lemma}

\begin{proof}
This follows immediately from the definition \eqref{eq:Takagi-def}, since the slope of $T_{2m}$ over the interval $[k/2^{2m},(k+1)/2^{2m}]$ is equal to $D_{2m}(x_0)=0$, and $T(x_0)=T_{2m}(x_0)$.
\end{proof}

\begin{definition} \label{def:humps}
{\rm
For a balanced dyadic rational $x_0=k/2^{2m}$ as in Lemma \ref{lem:similar-copies}, let $H(x_0)$ denote the portion of the graph of $T$ restricted to the interval $[k/2^{2m},(k+1)/2^{2m}]$. By Lemma \ref{lem:similar-copies}, $H(x_0)$ is a similar copy of the full graph $\mathcal{G}_T$; we call it a {\em hump}. Its height is $\frac23{(\frac14)}^m$, and we call $m$ its {\em order}. By the {\em generation} of the hump $H(x_0)$ we mean the generation of the balanced dyadic rational $x_0$. A hump of generation $1$ will be called a {\em first-generation hump}. By convention, the graph $\mathcal{G}_T$ itself is a hump of generation $0$. If $D_j(x_0)\geq 0$ for every $j\leq 2m$, we call $H(x_0)$ a {\em leading hump}. See Figure \ref{fig:humps} for an illustration of these concepts.
%We denote by $\HH=\{H(x_0): x_0\in\BB\}$ the set of all humps, and by $\HH'$ the subset of $\HH$ consisting of all leading humps.
}
\end{definition}

\begin{figure}
\begin{center}
\epsfig{file=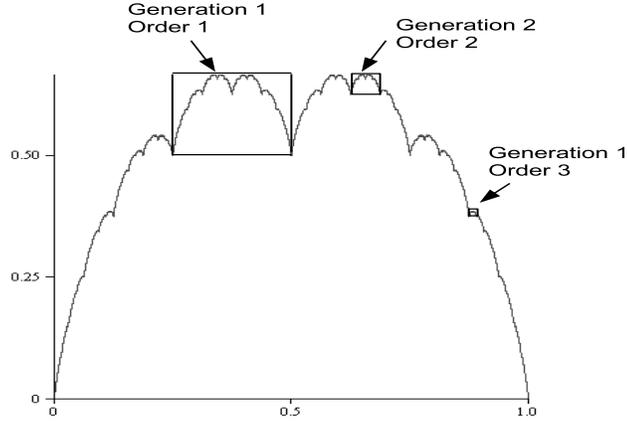, height=.35\textheight, width=.7\textwidth}
\caption{The ``humps" $H(1/4)$, $H(5/8)$ and $H(7/8)$, enclosed in rectangles from left to right. Note that in binary, $1/4=0.01$, $5/8=0.1010$, and $7/8=0.111000$. Only the first of these, $H(1/4)$, is a leading hump.}
\label{fig:humps}
\end{center}
\end{figure}

The Takagi function $T$ takes on a local maximum value at a point $x=\sum_{n=0}^\infty \eps_n/2^n$ precisely when the point $(x,T(x))$ is located at the top of some hump. This is the case if and only if for some $m\in\NN$, 
\begin{equation*}
\eps_1+\dots+\eps_{2m}=m, \quad\mbox{and}\quad \eps_{2n-1}+\eps_{2n}=1\quad \mbox{for each $n>m$}.
\end{equation*}
The first part of the above condition ensures that $(x,T(x))$ lies on a hump of order $m$; the second part implies that it lies at the top of that hump. In particular, the points of local maximum of $T$ lie dense in $[0,1]$. This result too is due to Kahane \cite{Kahane}. On the other hand, since $T_+'(x)=\infty$ and $T_-'(x)=-\infty$ at each dyadic $x$, $T$ has a local minimum value at every dyadic rational point $x$. Kahane shows that there are no local minima at non-dyadic points.

\subsection{Humps and Hausdorff measure} \label{subsec:Hausdorff}

It is often necessary to count the humps of a given order and/or generation, and this counting involves the {\em Catalan numbers}
\begin{equation*}
C_n:=\frac{1}{n+1}\binom{2n}{n}, \qquad n=0,1,2,\dots.
\end{equation*}

\begin{lemma} \label{lem:hump-count} 
Let $m\in\NN$.

(i) There are $\binom{2m}{m}$ humps of order $m$.

(ii) There are $C_m$ leading humps of order $m$.

(iii) There are $2C_{m-1}$ first-generation humps of order $m$.
\end{lemma}

This lemma is extremely helpful in the study of the level sets of $T$ (see Section \ref{sec:level-sets}). Another use is the following.
Mauldin and Williams \cite{Mauldin-Williams} first showed that the graph of $T$ has Hausdorff dimension one, but remarked that they did not know whether it has $\sigma$-finite linear measure. Anderson and Pitt \cite{Anderson} showed that the answer is affirmative, not only for the Takagi function but for a much wider class of functions (the so-called Zygmund space $\Lambda_d^*$). Odani \cite{Odani} explicitly decomposed the graph of $T$ into countably many sets of finite linear measure, as follows.

Let $S$ denote the set of points $(x,y)$ on $\GG_T$ which belong to humps of infinitely many generations, and for $n=0,1,2,\dots$, let $E_n$ denote the set of points which belong to a hump of generation $n$, but not to a hump of generation $n+1$. Then
\begin{equation*}
\GG_T=S\cup E_0\cup E_1\cup E_2\cup\dots.
\end{equation*}
Since $E_0$ is the graph of $T$ with all the first-generation humps removed, it is intuitively clear (and can be made rigorous) that the restriction of $T$ to $\pi_X(E_0)$ is monotone increasing on $[0,1/2]$, and monotone decreasing on $[1/2,1]$. Hence, $E_0$ has finite linear measure. Next, $E_1$ consists of countably many copies of $E_0$, one inside each first-generation hump. For each $m\in\NN$, there are $2C_{m-1}$ copies with contraction ratio $1/4^m$ by Lemma \ref{lem:hump-count}. Thus,
\begin{equation*}
\HH^1(E_1)=\sum_{m=1}^\infty2C_{m-1}\left(\frac14\right)^m\HH^1(E_0)=\frac12\sum_{n=0}^\infty C_n \left(\frac14\right)^n\HH^1(E_0)=\HH^1(E_0),
\end{equation*}
where we have used the well-known fact that $\sum_{n=0}^\infty C_n (1/4)^n=2$. Inductively, this argument can be continued to show that $\HH^1(E_n)<\infty$ for each $n$. It remains to verify that $\HH^1(S)<\infty$. Order the first-generation humps $H_1,H_2,\dots$ in some arbitrary manner, and for $i\in\NN$, let $\Phi_i$ be the similarity map which maps $\GG_T$ onto $H_i$. Then it is easy to check that
\begin{equation*}
S=\bigcup_{i=1}^\infty \Phi_i(S).
\end{equation*}
The open set condition is satisfied (take $(0,1)\times\RR$, say). As above, the contraction ratios of the $\Phi_i$ sum to 1, so Moran's equation gives $\dim_H S=1$. An easy exercise (it is clear which coverings to use) shows that $\HH^1(S)<\infty$. Thus, the graph of $T$ has $\sigma$-finite linear measure. (In fact, $\HH^1(S)>0$ as well, since $\pi_X(S)$ has full Lebesgue measure.)

Essentially the same construction is given by Buczolich \cite[Theorem 9]{Buczolich2}. He shows additionally that the set $S$ is an ``irregular" 1-set, meaning that it intersects every continuously differentiable curve in a set of $\HH^1$-measure zero. In \cite[Theorem 10]{Buczolich1}, Buczolich shows also that the Takagi function is ``micro self-similar", in the sense that the graph of $T$ itself is a micro tangent set of $T$ at almost every point $x\in[0,1]$.

\section{Alternative representations of $T(x)$} \label{sec:expressions}

While \eqref{eq:Takagi-def} is arguably the simplest and certainly the most common expression for $T(x)$, many other representations occur in the literature, and most have some unique advantage in proving certain things about the Takagi function or its generalizations. 

\bigskip
1. {\em Dynamical systems view}. To begin, put $\psi(x):=2\phi(x)$ for $x\in[0,1]$, and note that $\psi$ is a special case of a ``tent map", which maps $[0,1]$ onto itself. It is easy to see that we can write \eqref{eq:Takagi-def} as
\begin{equation}
T(x)=\sum_{n=1}^\infty \frac{1}{2^n}\psi^{(n)}(x),	\qquad x\in[0,1],
\label{eq:dynamics-def}
\end{equation}
where $\psi^{(n)}$ denotes $n$-fold iteration of $\psi$. Since $\sum_{n=1}^\infty 1/2^n=1$, \eqref{eq:dynamics-def} represents $T(x)$ as a weighted average of the iterates of $x$ under the chaotic dynamical system $\psi$.

\bigskip
2. {\em Takagi's definition}. Most authors define the Takagi function by either \eqref{eq:Takagi-def} or \eqref{eq:dynamics-def}, but it should be pointed out that Takagi himself defined $T(x)$ differently. For $n\in\NN$, let $a_n$ denote the number of binary digits among $\{\eps_1,\dots,\eps_{n-1}\}$ that are different from $\eps_n$. In other words, $a_n=O_n(x)$ if $\eps_n=1$, and $a_n=I_n(x)$ if $\eps_n=0$. Takagi defined $T(x)$ by
\begin{equation}
T(x)=\sum_{n=1}^\infty \frac{a_n}{2^n}.
\label{eq:Takagi-original-definition}
\end{equation}
To see that \eqref{eq:Takagi-def} and \eqref{eq:Takagi-original-definition} are equivalent, express $\phi(x)$ in terms of the binary expansion of $x$ by
\begin{equation*}
\phi(x)=\sum_{k=1}^\infty \frac{\eps_k(1-\eps_1)+(1-\eps_k)\eps_1}{2^k}. \end{equation*}
This generalizes to
\begin{equation}
\frac{\phi(2^n x)}{2^n}=\sum_{j=1}^\infty \frac{\eps_{n+j}(1-\eps_{n+1})+(1-\eps_{n+j})\eps_{n+1}}{2^{n+j}}.
\label{eq:phi-in-binary}
\end{equation}
We can similarly write $a_n=\eps_n O_n(x)+(1-\eps_n)I_n(x)$. Inserting \eqref{eq:phi-in-binary} into \eqref{eq:Takagi-def} and interchanging summations it is now easy to obtain \eqref{eq:Takagi-original-definition}. (This is essentially the proof given by Lagarias and Maddock \cite[Lemma 2.1]{LagMad1}.)

K\^ono \cite{Kono}, and later Gamkrelidze \cite{Gamkrelidze}, used a form similar to \eqref{eq:Takagi-original-definition} (expressing $a_n$ in terms of the Rademacher functions $X_n(x)=(-1)^{\eps_n}$) to investigate probabilistic properties of the graph of $T$.

\bigskip
3. {\em Tambs-Lyche's definition}. In 1939, Tambs-Lyche \cite{Tambs-Lyche} gave the following expression for $T(x)$. Write
\begin{equation*}
x=\sum_{j=1}^\infty 2^{-l_j},
\end{equation*}
where $\{l_j\}$ is a strictly increasing sequence of integers (the sum being finite if $x$ is dyadic). Then
\begin{equation}
T(x)=\sum_{j=1}^\infty\frac{l_j-2(j-1)}{2^{l_j}}.
\label{eq:Tambs-Lyche}
\end{equation}
This formula is useful for approximating solutions to the equation $T(x)=y$, as shown in \cite[Section 4.2]{Allaart5}.

Tambs-Lyche actually {\em defined} his function by the summation in the right hand side of \eqref{eq:Tambs-Lyche}, and stated without proof that it is equivalent to \eqref{eq:Takagi-def}. Tambs-Lyche's formula too has been rediscovered many times. In her thesis and subsequent publication \cite{Kawamura1}, Kawamura used a relationship with Lebesgue's singular function (explained in the next section) to obtain the expression
\begin{equation}
T(x)=\sum_{n=1}^\infty\eps_n(x)\frac{O_n(x)-I_n(x)+2}{2^n} =\sum_{n=1}^\infty\eps_n(x)\frac{n-2(I_n(x)-1)}{2^n},
\label{eq:Kawamura-form}
\end{equation}
which is clearly equivalent to \eqref{eq:Tambs-Lyche}. De Amo and Fern\'andez-S\'anchez \cite{deAmo} derive \eqref{eq:Tambs-Lyche} explicitly from \eqref{eq:Takagi-def}, while Kuroda \cite{Kuroda} deduces it directly from Takagi's definition \eqref{eq:Takagi-original-definition}.
Here we give a short proof using \eqref{eq:Takagi-def}. Recall the definition of $T_n$ from \eqref{eq:partial-Takagi}, and note that $T_n$ is piecewise linear with slope $D_n(x)$ at all points not of the form $j/2^n$. Moreover, $T(j/2^n)=T_n(j/2^n)$. Thus, if $0<l\leq m$ and $j\in\ZZ_+$, we have
\begin{equation}
T\left(\frac{j}{2^l}+\frac{1}{2^m}\right)-T\left(\frac{j}{2^l}\right)=\frac{1}{2^m}D_m\left(\frac{j}{2^l}\right)=\frac{m-2s_j}{2^m},
\label{eq:general-dyadic-difference}
\end{equation}
where $s_j$ is the number of $1$'s in the binary expansion of $j$.
This immediately gives $T(1/2^m)=m/2^m$, and a straightforward induction argument yields
\begin{equation*}
T\left(\sum_{j=1}^n 2^{-l_j}\right)=\sum_{j=1}^n \frac{l_j-2(j-1)}{2^{l_j}}
\end{equation*}
for all $n\in\NN$ and integers $1\leq l_1<l_2<\dots<l_n$. The continuity of $T$ gives \eqref{eq:Tambs-Lyche}.

As a by-product of \eqref{eq:general-dyadic-difference} (putting $l=m$ and summing over $j$), we obtain the formula given by Kr\"uppel \cite{Kruppel1}:
\begin{equation}
T\left(\frac{k}{2^m}\right)=\frac{1}{2^m}\sum_{j=0}^{k-1}(m-2s_j),
\label{eq:Kruppel}
\end{equation}
which was used in Section \ref{subsec:derivatives}.

\bigskip
4. {\em Random walk definition}. Lagarias and Maddock \cite[Section 2]{LagMad1} use \eqref{eq:Takagi-original-definition} to express $T(x)$ in terms of the sequence $\{D_n(x)\}$ as follows:
\begin{equation}
T(x)=\frac12-\frac14\sum_{n=1}^\infty(-1)^{\eps_{n+1}}\frac{D_n(x)}{2^n}.
\label{eq:D-expression}
\end{equation}
This formula is useful in the study of level sets, because one easily infers from it the following important fact.

\begin{lemma} \label{lem:invariance}
If $|D_n(x)|=|D_n(x')|$ for every $n$, then $T(x)=T(x')$.
\end{lemma}

Note that $\{D_n(x)\}_n$ is a symmetric simple random walk when $x$ is chosen at random in $[0,1]$. Thus, \eqref{eq:D-expression} expresses $T$ as a functional of a random walk.

\bigskip
5. {\em Fourier series}. While the definition \eqref{eq:Takagi-def} gives $T(x)$ directly as a Schauder series, it is also relatively easy to compute the Fourier series for $T(x)$. Hata and Yamaguti \cite{Hata-Yamaguti} point out that
\begin{equation*}
\phi(x)=\frac14-\frac{2}{\pi^2}\sum_{k\in\NN,\ k\ \mathrm{odd}}\frac{\cos 2\pi kx}{k^2},
\end{equation*}
and this results in the Fourier series
\begin{equation*}
T(x)=\frac12-\frac{2}{\pi^2}\sum_{m=1}^\infty A_m\cos 2\pi mx,
\end{equation*}
where $A_m=(2^n k^2)^{-1}$ if $m=2^n k$, with $k$ odd. The Fourier coefficients $A_m$ satisfy $1/m^2\leq A_m\leq 1/m$. In particular, the Fourier series of $T(x)$ is non-lacunary, in contrast to the Weierstrass function which is defined as a lacunary Fourier series.

\section{Functional and difference equations} \label{sec:FE}

De Rham~\cite{deRham1} was the first to point out that the Takagi function on $[0,1]$ satisfies the functional equation
\begin{equation} \label{eq:takagi FA}
       f(x) =
        \begin{cases}
                (1/2) f(2x)+ x,
                       & \qquad 0 \leq x \leq 1/2, \\
                (1/2) f(2x-1)+(1-x),
                        & \qquad 1/2 \leq x \leq 1.
        \end{cases}
\end{equation}
Kairies, Darsow and Frank~\cite{KaiDarFra} observed \eqref{eq:takagi FA} and proved the following results:
\begin{enumerate}\setlength{\itemsep}{0mm}
\item Any function $f:[0,1]\to\RR$ satisfying \eqref{eq:takagi FA} is nowhere differentiable, and coincides with $T$ on the dyadic rationals.
\item If $f:[0,1]\to\RR$ satisfies \eqref{eq:takagi FA}  and is bounded, then $f=T$.
\item A recursive relation for the moments $M_n=\int_0^1 x^n f(x)\,dx$ of any function $f$ satisfying \eqref{eq:takagi FA} is given by 
\begin{equation*}
M_0=1/2, \qquad M_n=\frac{1}{(n+1)(n+2)}+\frac{1}{2(2^{n+1}-1)}\sum_{k=0}^{n-1}\binom{n}{k} M_k.
\end{equation*} 
\end{enumerate}

(The paper \cite{KaiDarFra} was inadvertently printed before the page proofs were received; a list of corrections is given in \cite{Darsow}.) Later, extending the results of \cite{KaiDarFra}, Kairies~\cite{Kairies} gave a list of seven functional equations satisfied by $T(x)$, and investigated which subsets of these equations imply that a bounded function $f$ satisfying them must in fact be the Takagi function.

\bigskip

In 1984, Hata and Yamaguti~\cite{Hata-Yamaguti} started a new direction by regarding the Takagi function and related functions as solutions of discrete boundary value problems. This was quite natural, since \eqref{eq:Takagi-def} gives $T(x)$ directly as a Schauder series, and from the Schauder expansion of a function one quickly obtains an infinite system of difference equations which the function satisfies. Using these difference equations, Hata and Yamaguti showed that the Takagi function is closely related to another special function: Lebesgue's singular function. Kawamura~\cite{Kawamura1} later adopted their approach and found a close relationship between other nowhere differentiable functions, singular functions, and self-similar sets in the plane. 

First, we briefly recall Schauder expansions. Whereas a function's Fourier expansion uses trigometric functions, the Schauder expansion uses the ``tent" functions
\begin{equation*}
S_{n,i}(x)=\begin{cases}
2\phi(2^n x), & \mbox{if}\ \frac{i}{2^n}\leq x\leq\frac{i+1}{2^n}\\
0, & \mbox{otherwise},
\end{cases}
\end{equation*}
for $0\leq i\leq 2^n-1$ and $n\in\ZZ_+$. Thus, the graph of $S_{n,i}$ is the regular isosceles triangle of unit height whose base is the interval $[i/2^n,(i+1)/2^n]$. 

It is well known that every continuous function $f:[0,1]\to\RR$ which vanishes at $0$ and $1$ has a unique Schauder expansion of the form 
\begin{equation}
f(x)=\sum_{n=0}^{\infty}\sum_{i=0}^{2^n-1}a_{n,i}S_{n,i}(x),
\label{eq:Schauder-expansion}
\end{equation}
where
\begin{equation*}
a_{n,i}=f\left( \frac{2i+1}{2^{n+1}} \right)-\frac12 \left\{ f \left( \frac{i}{2^n} \right)+ f \left( \frac{i+1}{2^n} \right) \right\}.
\end{equation*}
Applying this to the Takagi function, we immediately obtain
 
\begin{theorem}[Hata-Yamaguti, 1983]
The Takagi function $T(x)$ is the unique continuous solution of the discrete boundary value problem
\begin{equation} \label{eq:DF-takagi}
 T\left( \frac{2i+1}{2^{n+1}} \right)-\frac12 \left\{ T \left( \frac{i}{2^n} \right)+ T \left( \frac{i+1}{2^n} \right) \right\}
 =\frac{1}{2^{n+1}},
\end{equation}
where $0 \le i\le 2^n-1$, $n\in \ZZ_+$, and the boundary conditions are $T(0)=T(1)=0$. 
\end{theorem}

Next, recall Lebesgue's singular function. Imagine flipping an unfair coin with probability $r \in (0,1)$ of heads and probability $1-r$ of tails. Note that $r \neq 1/2$. 
Let the binary expansion of $t \in [0,1]$: $t=\sum_{n=1}^{\infty}\omega_n / 2^n$ be determined by flipping the coin infinitely many times. 
More precisely, $\omega_n=0$ if the $n$-th toss is heads and $\omega_n=1$ if it is tails. 
We define {\it Lebesgue's singular function} $L_r(x)$ as the distribution function of $t$:
\begin{equation*}
L_r(x):=Prob\{t \leq x\}, \qquad 0 \leq x \leq 1.
\end{equation*} 
With the function $L_r$ is associated a probability measure $\mu_r$ on $[0,1]$, called the {\em binomial measure}, under which the binary digits of a number $t\in[0,1]$ are independent, taking the values $0$ and $1$ with probabilities $r$ and $1-r$, respectively.

It is well-known that $L_r(x)$ is strictly increasing, but its derivative is zero almost everywhere. De Rham~\cite{deRham2} showed that $L_r(x)$ is the unique continuous solution of the functional equation
\begin{equation} \label{eq:Lebesgue-FE}
        L_{r}(x)=
        \begin{cases}
                r L_{r}(2x), 
                        & \qquad 0 \leq x \leq \tfrac12, \\
                (1-r)L_r(2x-1)+r, 
                        & \qquad \tfrac12 \leq x \leq 1.
        \end{cases}
\end{equation}
Hata and Yamaguti showed that $L_r(x)$ is also the unique continuous solution of the following discrete boundary value problem:
\begin{equation} \label{eq:DF-lebesgue}
L_r\left( \frac{2i+1}{2^{n+1}} \right)=(1-r) L_r \left( \frac{i}{2^n} \right)+ r L_r \left(\frac{i+1}{2^n}\right),
\end{equation}
where $0 \le i\le 2^n-1$ and $n\in\ZZ_+$. The boundary conditions are $L_r(0)=0$ and  $L_r(1)=1$. From \eqref{eq:DF-takagi} and \eqref{eq:DF-lebesgue}, they proved the important and useful relationship
\begin{equation}
\left.\frac12\frac{\partial}{\partial r}L_r(x)\right|_{r=1/2}=T(x).
\label{eq:takagi-lebesgue}
\end{equation}
This identity can also be obtained from the following expression for $L_r(x)$, due to Lomnicki and Ulam~\cite{Lomnicki-Ulam-1934}:
\begin{equation} \label{eq:Ulam}
L_{r}(x)=\frac{r}{1-r}\sum_{n=1}^{\infty}\eps_n r^{n-I_n}(1-r)^{I_n}.
\end{equation}
Differentiating this with respect to $r$ and setting $r=1/2$ gives the right hand side of \eqref{eq:Kawamura-form}, and hence we have \eqref{eq:takagi-lebesgue}. (In \cite{Kawamura1} the reverse approach is taken, and \eqref{eq:Kawamura-form} is derived from \eqref{eq:Ulam} using \eqref{eq:takagi-lebesgue}.) We note that \eqref{eq:takagi-lebesgue} also leads to a very short proof of \eqref{eq:Kruppel}: It is easy to see that
\begin{equation*}
L_r\left(\frac{k}{2^m}\right)=\sum_{j=0}^{k-1}r^{m-s_j}(1-r)^{s_j},
\end{equation*}
where $s_j$ is the number of $1$'s in the binary expansion of $j$.
Differentiating gives \eqref{eq:Kruppel}.

The functional equations \eqref{eq:takagi FA} and \eqref{eq:Lebesgue-FE} are both special cases of the general family of functional equations studied by de Rham~\cite{deRham2}. De Rham considers the Takagi function and Lebesgue's singular function in separate papers, and does not appear to have noticed the relationship \eqref{eq:takagi-lebesgue}.

\subsection{Evaluating $T(x)$ for rational $x$} \label{subsec:rational}

One particular use of the functional equation \eqref{eq:takagi FA} is to the exact evaluation of $T(x)$ for rational $x$. As noted by Knuth \cite[p.~32, p.~103]{Knuth}, $T(x)$ is rational whenever $x$ is, and by applying \eqref{eq:takagi FA} repeatedly one obtains a system of linear equations which is easily solved. We give the details here, and also examine the number of iterations required to compute $T(x)$.

Let $x=p/q$, where $p,q\in\NN$ with $gcd(p,q)=1$. Assume first that $p<q/2$. Then, by \eqref{eq:takagi FA} and the symmetry of $T$, we can write $T(p/q)=\frac12 T(p'/q)+(p/q)$, where $p'=\min\{2p,q-2p\}$. If $q$ is even, the fraction $p'/q$ simplifies. If $q$ is odd, then $gcd(p',q)=1$ again. These ideas lead to the following two-stage algorithm for evaluating $T(p/q)$:

\bigskip
{\em Step 1.} Let $q=2^m q'$, with $q'$ odd, and assume $gcd(p,q)=1$. Put $q_0:=q$, $p_0:=\min\{p,q-p\}$, and $q_j:=q_{j-1}/2$, $p_j:=\min\{p_{j-1},q_j-p_{j-1}\}$, $j=1,\dots,m$. Let $p':=p_m$. Then, after applying the functional equation $m$ times, putting the results together and simplifying, we obtain
\begin{equation*}
T\left(\frac{p}{q}\right)=\frac{1}{2^m}T\left(\frac{p'}{q'}\right)+\frac{1}{q}\sum_{j=0}^{m-1}p_j.
\end{equation*}
So it remains to compute $T(p/q)$ for odd $q$.

\bigskip
{\em Step 2.} Let $q$ be odd and $gcd(p,q)=1$. Put $p_0:=\min\{p,q-p\}$, and $p_j:=\min\{2p_{j-1},q-2p_{j-1}\}$, $j=1,2,\dots$. Note that $T(p_j/q)=\frac12 T(p_{j+1}/q)+(p_j/q)$ for each $j\geq 0$. Since $p_j\equiv \pm 2p_{j-1}\, (\mod q)$, we have $p_j\equiv \pm 2^j p_0\, (\mod q)$ for each $j$, so there will be some positive integer $j$ such that $p_j=p_0$. The smallest such $j$ is the number $k:=\min\{j\in\NN: 2^j\equiv \pm 1\, (\mod q)\}$. Such a $j$ always exists by Euler's theorem, and $k\leq \varphi(q)$, where the Euler function $\varphi(q)$ denotes the number of integers in $\{1,2,\dots,q-1\}$ relatively prime to $q$. Let $\ord_q(2)$ denote the order of $2$ in the group of units of $\ZZ/q\ZZ$. That is, $\ord_q(2)$ is the smallest positive integer $n$ such that $2^n\equiv 1\, (\mod q)$. It is an elementary exercise in number theory to show that
\begin{equation*}
k=\begin{cases}
\frac12\ord_q(2), & \mbox{if $q|2^j+1$ for some $j\in\NN$}\\
\ord_q(2), & \mbox{otherwise}.
\end{cases}
\end{equation*}
It now takes precisely $k$ iterations of the functional equation to express $T(p/q)$ in terms of itself. Solving for $T(p/q)$ and simplifying, we eventually obtain
\begin{equation*}
T\left(\frac{p}{q}\right)=\frac{1}{q(2^k-1)}\sum_{j=0}^{k-1}2^{k-j}p_j.
\end{equation*}

The inverse problem is also interesting: given a rational $y$ in the range of $T$, is there a rational point $x\in[0,1]$ such that $T(x)=y$? This question is still open. Knuth \cite[p.~103]{Knuth} gives an algorithm which produces solutions of the equation $T(x)=y$ for many rational $y$. This involves starting with an initial value $v$ and walking along a particular directed graph, updating $v$ by some fixed arithmetic operation at each node. Which node may be visited next depends not only on the graph but also on whether the transition will keep the value of $v$ within certain bounds. The algorithm terminates when one visits a node for the second time with the same value of $v$. It is, however, not known whether the algorithm always terminates. A different method which gives solutions for many (but not all) rational $y$ is given by Allaart \cite[Section 4.2]{Allaart5}.

\section{Level sets} \label{sec:level-sets}

In this section we consider the level sets
\begin{equation*}
L(y):=\{x\in[0,1]: T(x)=y\}, \qquad y\in\RR.
\end{equation*}
Of course, $L(y)=\emptyset$ if $y\not\in[0,2/3]$. The simplest level set is $L(0)=\{0,1\}$. At the other extreme, in view of Theorem \ref{thm:maxima}, we have that $L(2/3)$ is an uncountable (Cantor) set of dimension $1/2$. In general, $L(y)$ can be finite, countably infinite or uncountable, and which of these three possibilities is the most common depends on the precise mathematical meaning assigned the word ``most". If $L(y)$ is finite, it must have an even number of points by the symmetry of the graph of $T$, but any even positive number is possible. If $L(y)$ is uncountable, its Hausdorff dimension can be zero or strictly positive, but never more than $1/2$.

Level sets are partitioned into easier to understand pieces called {\em local level sets}. Each local level set is either finite or a Cantor set, and the members of a local level set are easily obtained from one another by certain combinatorial operations (``block flips") on their binary expansions. A level set can consist of finitely many, countably infinitely many, or uncountably many local level sets. As with the cardinalities of the level sets, which of these three is the most common depends on how one defines ``most". Many questions about the level sets of $T$ remain open.

\subsection{Finite or infinite?} \label{subsec:infinite}

The level sets of the Takagi function and related functions were first considered by Anderson and Pitt \cite{Anderson}. Their Theorem 7.3, which applies to a large class of functions, implies that $L(y)$ is countable for almost every $y\in[0,\frac23]$. This result was improved recently by Buczolich \cite{Buczolich2}, who focused on the Takagi function itself and concluded the following.

\begin{theorem}[Buczolich 2008] \label{thm:buczo}
For almost every ordinate $y$, $L(y)$ is finite.
\end{theorem}

\begin{proof}[Sketch of proof (Allaart 2011)]
We sketch a proof here that is slightly different from the original proof given by Buczolich, and which uses the notion of humps and leading humps defined in Section \ref{sec:graph}. For the full details, see \cite[Section 3]{Allaart4}. 

Observe first that Lemma \ref{lem:invariance} immediately implies the following.

\begin{lemma} \label{lem:leading-humps}
For every hump $H$ there is a leading hump $H'$ of the same order and generation as $H$, such that $\pi_Y(H)=\pi_Y(H')$. On the other hand, for every leading hump $H'$ there are only finitely many humps $H$ such that $\pi_Y(H)=\pi_Y(H')$.
\end{lemma}

Next, define a set
\begin{equation*}
X^*:=[0,1]\backslash\bigcup_{x_0\in\BB_1}I(x_0).
\end{equation*}
In other words, $X^*$ is obtained by removing all the dyadic closed intervals above which the graph of $T$ has a first-generation hump. The importance of $X^*$ is made clear by the next lemma, which we state here without proof.

\begin{lemma} \label{lem:bijection}
The Takagi function $T$ maps $X^*$ onto $[0,\frac12]$. Moreover, $T$ is strictly increasing on $X^*\cap[0,\frac12)$.
\end{lemma}

Lemmas \ref{lem:leading-humps} and \ref{lem:bijection} can be used to prove the crucial fact that $L(y)$ is finite whenever the horizontal line $l_y$ at level $y$ intersects only finitely many leading humps. Using this fact, the proof of the theorem can be completed as follows. If $y$ is chosen at random from $[0,\frac23]$ and $H$ is a leading hump of order $m$, the probability that the line $l_y$ intersects $H$ is ${(\frac14)}^m$. Letting $\HH'$ denote the set of all leading humps, this gives
\begin{equation*}
\sum_{H\in\HH'}\sP(y\in \pi_Y(H))=\sum_{m=0}^\infty C_m\left(\frac14\right)^m<\infty,
\end{equation*}
since there are $C_m$ leading humps of order $m$ by Lemma \ref{lem:hump-count}, and $C_m\sim 4^m/(m^{3/2}\sqrt{\pi})$. Thus, by the Borel-Cantelli lemma, the probability that $l_y$ intersects infinitely many leading humps is zero. Therefore, $L(y)$ is finite with probability $1$.
\end{proof}

Despite the above result, the {\em average} cardinality of the level sets of $T$ is infinite. That is,
\begin{equation*}
\int_0^{2/3} |L(y)|=\infty.
\end{equation*}
This was shown by Lagarias and Maddock \cite{LagMad2}. An alternative proof based on Lemma \ref{lem:bijection} is given by Allaart \cite{Allaart4}.

Whereas the above results are probabilistic in nature, a quite different picture emerges when one views the level sets of $T$ from the perspective of Baire category. Define the sets
\begin{gather*}
S_\infty^{co}:=\{y\in[0,\tfrac23]: L(y)\ \mbox{is countably infinite}\},\\
S_\infty^{uc}:=\{y\in[0,\tfrac23]: L(y)\ \mbox{is uncountably infinite}\}.
\end{gather*}

\begin{theorem}[Allaart 2011] \label{thm:S-uncountable}
The set $S_\infty^{uc}$ has the decomposition
%\begin{equation}
$S_\infty^{uc}=E\cup M$,
%\end{equation}
where $E$ is a dense $G_\delta$ set, and $M$ is a countable set disjoint from $E$ which consists exactly of the local maximum ordinate values of $T$. As a result, the set
%\begin{equation*}
$\{y\in[0,\frac23]: L(y)\ \mbox{\rm is countable}\}$
%\end{equation*}
is of the first category.
\end{theorem}

For the proof and a more explicit description of the set $E$, see Allaart \cite[Section~4]{Allaart4}. There it is also shown that $S_\infty^{uc}$ does not contain any dyadic rational ordinates $y$. Note that the residual set $S_\infty^{uc}$ has Lebesgue measure zero by Theorem \ref{thm:buczo}. However, it has full Hausdorff dimension one; see Lagarias and Maddock \cite{LagMad2}.

The set $S_\infty^{co}$ is rather more difficult to describe. It contains the images of all dyadic rational abscissas $x$ in $[0,1]$, so it is dense in $[0,\frac23]$. But it is not known whether it contains more. 

\bigskip
Closely related to the level sets of $T$ is the {\em occupation measure} defined by $\mu_T(A)=\lambda(\{x\in[0,1]: T(x)\in A\})$, for Borel sets $A\subset\RR$. Buczolich \cite{Buczolich2} shows that $\mu_T$ is singular with respect to Lebesgue measure. This is witnessed by the set $S$ from Section \ref{subsec:Hausdorff}: It is relatively straightforward to show that $|L(y)|=\infty$ when $y\in \pi_Y(S)$, so Theorem \ref{thm:buczo} implies $\lambda(\pi_Y(S))=0$. On the other hand, $\pi_X(S)$ has full measure in $[0,1]$, because almost every $x\in[0,1]$ has the property that $D_n(x)=0$ for infinitely many $n$. Consequently, $\mu_T(\pi_Y(S))=\lambda(\pi_X(S))=1$.

\subsection{Cardinalities of finite level sets} \label{subsec:finite}

Since almost all level sets are finite in the Lebesgue sense, it is natural to ask what these finite cardinalities can be. By the symmetry of $T$ and the fact that $L(T(\frac12))=L(\frac12)$ is countably infinite, the number of points in each finite level set must be even. Allaart \cite{Allaart5} shows that, vice versa, every positive even number occurs. In fact, we have the following. Let
\begin{equation*}
S_{2n}:=\{y: |L(y)|=2n\}, \qquad n\in\NN.
\end{equation*}

\begin{theorem}[Allaart 2011] \label{thm:all-evens}
For each $n\in\NN$, $S_{2n}$ is uncountable but nowhere dense.
\end{theorem}

It is not known whether $S_{2n}$ has positive Lebesgue measure for each $n$. The argument used in the proof of Theorem \ref{thm:all-evens} unfortunately does not give enough to prove this stronger statement. As a partial result in this direction, however, Allaart \cite{Allaart5} shows that $\lambda(S_{2n})>0$ whenever $n$ is either a power of 2, or the sum or difference of two powers of 2. And for the specific case of $S_2$, fairly tight bounds on its Lebesgue measure can be given which show, not surprisingly perhaps, that $2$ is the most common finite cardinality.

\begin{theorem}[Allaart 2011] \label{thm:measure-of-S2}
The Lebesgue measure $\lambda(S_2)$ of $S_2$ satisfies
\begin{equation*}
\frac{5}{12}<\lambda(S_2)<\frac{35}{72}.
%\label{eq:S2-sandwich}
\end{equation*}
\end{theorem}

Since the graph of $T$ has height $2/3$, this implies that if an ordinate $y$ is chosen at random in the range of $T$, the probability that $L(y)$ contains exactly two points lies between 62.5\% and 72.9\%. The proof of Theorem \ref{thm:measure-of-S2} is based on a simple counting argument which uses the fact, from Lemma \ref{lem:hump-count}, that the graph of $T$ contains exactly $C_{m-1}$ first-generation leading humps. See \cite{Allaart5} for the details.

Which specific ordinates $y$ satisfy $|L(y)|=2$? It is clear from the graph that $y$ must be less than $\frac12$. Allaart \cite{Allaart5} gives two sufficient conditions, the first condition being directly in terms of the binary expansion of $y$.

\begin{theorem}[Allaart 2011] \label{thm:no-3-zeros}
Let $0<y<\frac12$ such that $y$ is not a dyadic rational, and suppose the binary expansion of $y$ does not contain a string of three consecutive $0$'s anywhere after the occurrence of its first $1$. Then $|L(y)|=2$.
\end{theorem}

Thus, for instance, the set $S_2$ includes the points $y=1/3=0.(01)^\infty$, $y=1/7=0.(001)^\infty$, and $y=1/40=0.0^5(1100)^\infty$. 

The second condition, which is neither weaker nor stronger than the first, involves the orbit of $y$ under iteration of the map
\begin{equation*}
\Psi(x)=\begin{cases}
0, & \mbox{if $y=0$}\\
4\left(y-\frac{k}{2^k}\right), & \mbox{if} \quad \frac{k}{2^k}\leq y<\frac{k-1}{2^{k-1}}, \quad k=3,4,\dots.
\end{cases}
\end{equation*}
Note that $\Psi$ maps $[0,\frac12)$ onto itself. Let $\Psi^n$ denote the $n$th iterate of $\Psi$, with $\Psi^0(y):=y$.
For $n=0,1,2,\dots$, let $k_n$ be that number $k\geq 3$ for which $k/2^k\leq \Psi^n(y)<(k-1)/2^{k-1}$, or put $k_n=\infty$ if $\Psi^n(y)=0$.

\begin{theorem}[Allaart 2011] \label{thm:double-at-most}
Let $0\leq y<\frac12$. If $k_{n+1}\leq 2k_n$ for every $n$, then $|L(y)|=2$.
\end{theorem}

This gives many more examples. For instance, the condition of the theorem obviously holds for the fixed points of $\Psi$, which are $y_k^*:=k/(3\cdot 2^{k-2})$, $k\geq 4$. Perhaps unexpectedly, this shows that infinitely many dyadic rational ordinates belong to $S_2$, the first three being $1/8$, $3/2^7$ and $1/2^8$. More generally, of course, many of the periodic points of $\Psi$ satisfy $k_{n+1}\leq 2k_n$ and are therefore in $S_2$. For instance, $y=1/11$ does not satisfy the ``no 3 zeros" condition of Theorem \ref{thm:no-3-zeros}, but it has $(k_n)_{n\geq 0}={(7,6,5,5,6,5,5,4,4,4)}^\infty$, and since $7\leq 2\cdot 4$, Theorem \ref{thm:double-at-most} yields $1/11\in S_2$.

Allaart \cite{Allaart5} actually gives a slightly weaker condition than the one given in Theorem \ref{thm:double-at-most}, which includes lower order terms, and also gives an accompanying necessary condition in terms of the sequence $(k_n)$. Together these conditions cover most cases, but still leave a small gap.

Theorems \ref{thm:no-3-zeros} and \ref{thm:double-at-most} are corollaries to the following, precise but somewhat abstract, characterization of membership in $S_2$.

\begin{theorem}[Allaart 2011] \label{thm:when-just-two}
Let $0<y<\frac12$. Then $|L(y)|=2$ if and only if
\begin{equation*}
\Phi(\Psi^n(y))>\frac23 \qquad\mbox{for all $n\geq 0$},
%\label{eq:two-condition}
\end{equation*}
where
\begin{equation*}
\Phi(y):=\begin{cases}
0, & \mbox{if $y=0$},\\
4^k(y-t_k), & \mbox{if} \quad \frac{k}{2^k}\leq y<\frac{k-1}{2^{k-1}}, \quad k=3,4,\dots.
\end{cases}
\end{equation*}
\end{theorem}

\subsection{Hausdorff dimension}

Recall from Section \ref{sec:graph} that the maximum value of $T$ is $2/3$, and $L(2/3)$ is a Cantor set of Hausdorff dimension $1/2$. One might ask whether there exist any level sets with Hausdorff dimension strictly greater than $1/2$. This question was first addressed by Maddock \cite{Maddock}, who proved that the intersection of the graph of $T$ with {\em any} line of integer slope has Hausdorff dimension at most $0.668$. In particular, $0.668$ is an upper bound for the dimensions of the level sets. Maddock himself conjectured that the real maximum is $1/2$. The issue was finally settled by de Amo et al. \cite{ABDF}.

\begin{theorem}[de Amo et al. 2011]
For each ordinate $y$, the box-counting dimension of $L(y)$ is at most $1/2$, and hence, $\dim_H L(y)\leq 1/2$.
\end{theorem}

The proof, which is surprisingly elementary, makes good use of the self-affinity of the graph of $T$ and uses a cleverly devised induction argument.

\bigskip
A direct consequence of Theorem \ref{thm:buczo} is that almost all level sets (in the Lebesgue sense) have Hausdorff dimension zero. On the other hand, Lagarias and Maddock \cite{LagMad2} have shown that the set
\begin{equation*}
\{y: \dim_H L(y)>0\}
\end{equation*}
has full Hausdorff dimension one. This is accomplished by putting a sequence of subsets of the range $[0,\frac23]$ in one-to-one bi-Lipschitz correspondence with certain well-behaved subsets of the domain $[0,1]$ whose Hausdorff dimension is easy to calculate and gets arbitrarily close to 1.

It is not known exactly which numbers occur as the Hausdorff dimension of some level set of $T$.

\subsection{Local level sets} \label{subsec:local}

Lagarias and Maddock \cite{LagMad1,LagMad2} introduce the concept of a local level set of the Takagi function. They first define an equivalence relation on $[0,1]$ by
\begin{equation}
x\sim x'\quad \stackrel{\defin}{\Longleftrightarrow}\quad |D_j(x)|=|D_j(x')|\ \mbox{for each $j\in\NN$}.
\label{eq:equivalence-relation}
\end{equation}
The {\em local level set} containing $x$ is defined by 
\begin{equation*}
L_x^{loc}:=\{x': x'\sim x\}.
\end{equation*}
Note that by Lemma \ref{lem:invariance}, $x\sim x'$ implies $T(x)=T(x')$, so each local level set is contained in some level set. Lagarias and Maddock point out that each local level set is either finite or a Cantor set. Members of the same local level set can be obtained from one another by simple operations on their binary expansions, called ``block flips" in \cite{LagMad1}. This works as follows. Let $Z(x)=\{n\geq 0: D_n(x)=0\}\cup\{\infty\}$. 
For any two elements $k,l\in Z(x)$ with $k<l\leq\infty$, form the point $x'$ with binary expansion $x'=\sum_{n=1}^\infty 2^{-n}\eps_n'$ by setting $\eps_n'=\eps_n$ if $n\leq k$ or $n>l$, and $\eps_n'=-\eps_n$ if $k<n\leq l$. Then $|D_n(x')|=|D_n(x)|$ for each $n$, and so $x'\in L_x^{loc}$. Every element of $L_x^{loc}$ can be obtained from $x$ by at most countably many operations of this type.

One of the results in \cite{LagMad1} concerns the average number of local level sets contained in a level set chosen at random. Let $N^{loc}(y)$ denote the number of local level sets contained in $L(y)$. 

\begin{remark}
{\rm
Lagarias and Maddock \cite{LagMad1} define $L_x^{loc}$ slightly differently, effectively viewing local level sets as subsets of the Cantor space $\{0,1\}^\NN$. However, this distinction does not affect the number of local level sets contained in any level set, which is all we are concerned with in this survey.
}
\end{remark}

\begin{theorem}[Lagarias and Maddock, 2010] \label{thm:local-level-sets}
The expected number of local level sets contained in a level set $L(y)$ with $y$ chosen at random from $[0,\frac23]$ is $\frac32$. More precisely,
\begin{equation*}
\rE[N^{loc}(y)]:=\frac32\int_0^{2/3}N^{loc}(y)\,dy=\frac32.
\end{equation*}
\end{theorem}

A simpler proof of this theorem is given by Allaart \cite{Allaart4}. In that paper, local level sets are also examined from the category point of view. The result is in marked contrast with the conclusion of Theorem \ref{thm:local-level-sets}. Define the sets
\begin{gather*}
S_\infty^{loc}:=\{y: L(y)\ \mbox{\rm contains infinitely many different local level sets}\},\\
S_\infty^{loc,uc}:=\{y: L(y)\ \mbox{\rm contains uncountably many different local level sets}\}.
\end{gather*}

\begin{theorem}[Allaart 2011]
\label{thm:infinite-local-level-sets}
(i) The set $S_\infty^{loc}$ is residual (co-meager) in $[0,\frac23]$.

(ii) The set $S_\infty^{loc,uc}$ is dense in $[0,\frac23]$, and intersects any subinterval of $[0,\frac23]$ in a continuum. 
\end{theorem}

\subsection{Open problems} \label{subsec:open}

A number of interesting questions about the level sets of the Takagi function remain open. We give a brief selection here, and refer to Lagarias \cite{Lagarias} for additional problems.

\begin{problem}
{\rm
Describe the set $S_\infty^{co}$ of ordinates $y$ with a countably infinite level set. Or less ambitiously, determine whether $S_\infty^{co}$ contains any points which are not the image of a dyadic rational.
}
\end{problem}

\begin{problem}
{\rm
(Knuth \cite[p.~32, Exer.~83]{Knuth}) Does there exist, for each rational ordinate $y\in[0,\frac23]$, a rational abscissa $x\in[0,1]$ such that $T(x)=y$? If $x$ is irrational but $y=T(x)$ is rational, must $L(y)$ be uncountable?
}
\end{problem}

\begin{problem}
{\rm
Determine the probability distribution of $|L(y)|$. In other words, find $\lambda(S_{2n})$ for each $n\in\NN$. This could be very difficult, but the weaker problem of showing that $\lambda(S_{2n})>0$ for every $n$ (or finding a counterexample) may be solvable.
}
\end{problem}

\begin{problem}
{\rm
Determine the probability distribution of $N^{loc}(y)$, the number of local level sets contained in $L(y)$. This too may be difficult.
}
\end{problem}

\begin{problem}
{\rm
The set $S_\infty^{loc,uc}$ intersects each subinterval of $[0,\frac23]$ in a continuum. Is it residual? Does it have Hausdorff dimension 1? Weaker than the last question, does $S_\infty^{loc}$ have Hausdorff dimension 1?
}
\end{problem}

\begin{problem}
{\rm
(Lagarias \cite{Lagarias}) Determine the dimension spectrum of $T$. That is, determine the function
\begin{equation*}
f(\alpha):=\dim_H \{y: \dim_H L(y)\geq\alpha\}.
\end{equation*}
}
\end{problem}

\section{Generalizations} \label{sec:generalizations}

\subsection{The Takagi-van der Waerden functions} \label{subsec:vdW}

An immediate generalization of the Takagi function is the sequence of functions
\begin{equation*}
f_r(x):=\sum_{n=0}^\infty \frac{1}{r^n}\phi(r^n x),  \qquad r=2,3,\dots.
\end{equation*}
Thus, $f_2$ is the Takagi function and $f_{10}$ is van der Waerden's function. Billingsley's argument (with only trivial modifications) shows that each $f_r$ is nowhere differentiable. Van der Waerden's original and elegant proof for the case $r=10$ works for all even $r\geq 4$, but not for odd $r$ or for $r=2$.

Each $f_r$ is also H\"older continuous of any order $\alpha<1$. This follows from the more general result of Shidfar and Sabetfakhri \cite{Shidfar2}, but is also easy to prove directly. In fact, by analogy with \eqref{eq:max-oscillations}, we have
\begin{equation*}
-1\leq\liminf_{h\to 0} \frac{f_r(x+h)-f_r(x)}{h\log_r(1/|h|)}\leq \limsup_{h\to 0} \frac{f_r(x+h)-f_r(x)}{h\log_r(1/|h|)}\leq 1.
\end{equation*}
It seems likely that $f_r$ has an infinite derivative at many points, but we have not found any detailed study of this question in the literature.

Generalizing Kahane's result \cite{Kahane}, Baba \cite{Baba} determines for each $r\geq 2$ the maximum value $M_r$ of $f_r$ and the set of points $E_r=\{x\in[0,1]: f_r(x)=M_r\}$.

\begin{theorem}[Baba 1984]
(i) If $r$ is odd, then $E_r=\{1/2\}$ and $M_r=r/(2r-2)$.

(ii) If $r$ is even, then $E_r$ is a Cantor set of dimension $1/2$, and $M_r=r^2/(2r^2-2)$.
\end{theorem}

\subsection{The Takagi class} \label{subsec:Takagi-class}

Another direct generalization of the Takagi function is obtained by replacing the factor $1/2^n$ in \eqref{eq:Takagi-def} with a general (real) constant $c_n$. This gives functions of the form
\begin{equation}
f(x)=\sum_{n=0}^\infty c_n\phi(2^n x).
\label{eq:Takagi-class}
\end{equation}
It is immediately clear that the series converges uniformly, and hence defines a continuous function $f$, when $\sum_{n=0}^\infty |c_n|<\infty$. Hata and Yamaguti \cite{Hata-Yamaguti} show that this condition is also necessary, and call the collection of functions of the form \eqref{eq:Takagi-class} the {\em Takagi class}. They note that each member of the Takagi class solves a discrete version of the Dirichlet boundary value problem involving the ``discrete Laplacian"
\begin{equation*}
\Delta_{i,n}f:=f\left(\frac{i}{2^{n}}\right)+ f\left(\frac{i+1}{2^{n}}\right)- 2f\left(\frac{2i+1}{2^{n+1}}\right),
\end{equation*}
in the sense that $\Delta_{i,n}f=-c_{n}$ for $n\geq 0$ and $i=0,1,\dots,2^{n}-1$, with $f(0)=f(1)=0$.

Beside the Takagi function itself, the Takagi class contains a number of interesting examples from the classical literature. For instance, Faber \cite{Faber} introduced the highly lacunary series
\begin{equation*}
F(x)=\sum_{k=1}^\infty \frac{1}{10^k}\phi(2^{k!}x),
\end{equation*}
and showed that $F$ has no derivative, finite or infinite, at any point. Moreover, he proved that $F$ does not satisfy a Lipschitz condition of any order. Kahane likewise constructs lacunary series of the form $f(x)=\sum_{\nu=1}^\infty p_\nu 2^{-k_\nu}\phi(2^{k_\nu}x)$, and shows how $p_\nu$ and $k_\nu$ can be chosen so that the modulus of continuity of $f$ is majorized (respectively minorized) by a given function satisfying appropriate conditions. 

Many of the functions in the Takagi class are fractals, in the sense that the Hausdorff dimension of their graph is strictly greater than one. Besicovitch and Ursell \cite{Besicovitch} showed that, if $f$ is any function satisfying a Lipschitz condition of order $\delta\in(0,1]$, then
\begin{equation}
1\leq \dim_H(f)\leq 2-\delta,
\label{eq:dim-estimates}
\end{equation}
where $\dim_H(f)$ denotes the Hausdorff dimension of the graph of $f$. They show that within these bounds any dimension is possible. The implications of their work for the Takagi class are collected in the following theorem.

\begin{theorem}[Besicovitch and Ursell, 1937] \label{thm:BU}
Let $f(x)=\sum_{n=0}^\infty 2^{-\delta a_n}\phi(2^{a_n}x)$, where $0<\delta<1$ and $a_{n+1}-a_n\geq A>0$. Then $f$ is Lipschitz of order $\delta$ but of no smaller order, and:

(i) If $a_{n+1}/a_n\to\infty$, then $\dim_H(f)=1$.

(ii) If $a_n=\mu^n$ where $\mu>1$, then $1<\dim_H(f)<2-\delta$. Moreover, for each $d\in(1,2-\delta)$ there exists $\mu>1$ such that, if $a_n=\mu^n$, then $\dim_H(f)=d$.

(iii) If $a_{n+1}/a_n\to 1$ but $a_{n+1}-a_n\to\infty$, then $\dim_H(f)=2-\delta$.
\end{theorem}

Examples of (i), (ii) and (iii) are, respectively: $a_n=2^{n^2}$, $a_n=2^n$, and $a_n=n^2$. Note that in all three cases the series \eqref{eq:Takagi-class} is lacunary. Generally speaking, the more lacunary the series, the smaller the dimension of the graph is. When the series is extremely lacunary as in (i), the fine structure of the graph virtually disappears and the function $f$ becomes ``almost differentiable". But the above theorem does not say anything about the important (nonlacunary) case $a_n=n$. In other words, it does not give the dimension of the functions
\begin{equation}
g_\delta(x)=\sum_{n=0}^\infty 2^{-\delta n}\phi(2^n x), \qquad 0<\delta<1.
\label{eq:Ledrappier}
\end{equation}
This boundary case was addressed more than half a century later by Ledrappier, using modern tools that were not yet available to Besicovitch and Ursell. 

Here too $g_\delta$ is Lipschitz of order $\delta$. Ledrappier showed that typically, $g_\delta$ attains the upper bound in \eqref{eq:dim-estimates}. More precisely, he proved that $\dim_H(g_\delta)=2-\delta$ whenever $2^{\delta-1}$ is an Erd\H os number. A number $\lambda\in(0,1)$ is called an Erd\H os number if the probability distribution of $\sum_{n=0}^\infty\lambda^n\eps_n$ (the so-called {\em Bernoulli convolution}) has Hausdorff dimension 1, where $\{\eps_n\}$ are i.i.d. random variables taking the values $1$ and $-1$ each with probability $1/2$. Later progress on Bernoulli convolutions due to Solomyak \cite{Solomyak} implies that almost every $\lambda\in(1/2,1)$ is Erd\H os, and hence, by Ledrappier's result, $\dim_H(g_\delta)=2-\delta$ for almost every $\delta\in(0,1)$.
Ledrappier's approach is dynamical, viewing the graph of $g_\delta$ as the repeller for some expanding self-map of $[0,1]\times\RR$. The deep and difficult proof uses ideas from smooth ergodic theory and a Marstrand-type lemma concerning projections of measures.

Recently, Katzourakis \cite{Katzourakis} has reported that, with $a_n=2\nu n$ for a parameter $\nu\in\NN$, the function $f$ in Theorem \ref{thm:BU} can be used to construct a ``pathological" solution to the nonlinear Aronsson partial differential equation and to the Infinity-Laplace PDE system.

For $\delta>1$, the function $g_\delta$ defined by \eqref{eq:Ledrappier} is Lipschitz and hence almost everywhere differentiable. When $1\leq\delta\leq 2$, $g_\delta$ turns out to be the extremal function in a certain approximate convexity problem; see Section \ref{subsec:app-analysis} below. 

\bigskip
After the publication of Hata and Yamaguti's paper, K\^ono \cite{Kono} investigated the Takagi class in greater generality. Perhaps the most striking result, concerning the differentiability of $f$, is the following:

\begin{theorem}[K\^ono 1987] \label{thm:Kono-differentiability}
Let $f$ be defined by \eqref{eq:Takagi-class}, and put $a_n:=2^n c_n$.

{\rm (i)} If $\{a_n\}\in\ell^2$, then $f$ is absolutely continuous and hence differentiable almost everywhere.

{\rm (ii)} If $\{a_n\}\not\in\ell^2$ but $\lim_{n\to\infty}a_n=0$, then $f$ is nondifferentiable at almost every point of $[0,1]$, but $f$ is differentiable on an uncountably large set, and the range of $f'$ is $\RR$.

{\rm (iii)} If $\limsup_{n\to\infty}|a_n|>0$, then $f$ is nowhere differentiable.
\end{theorem}

K\^ono \cite{Kono} also considers the oscillations of $f$ (stating more general forms of Theorems \ref{thm:uniform-modulus} and \ref{thm:local-modulus}), and proves furthermore that the Takagi class contains only one function which is smooth in Zygmund's sense. That is, if
\begin{equation}
f(x+h)+f(x-h)-2f(x)=o(h) \qquad\mbox{as $h\downarrow 0$}
\label{eq:Zygmund-smooth}
\end{equation}
for all $x\in(0,1)$, then $c_n=a/4^n$ for some constant $a$, and $f(x)=2ax(1-x)$.

\bigskip
A special case of the Takagi class arises when one takes $c_n=\pm 1/2^n$ for all $n$ in \eqref{eq:Takagi-class}. Precisely, let ${\bf r}=(r_0,r_1,\dots)$ be a sequence with $r_n\in\{-1,1\}$ for each $n$, and define
\begin{equation*}
F_{\bf r}(x)=\sum_{n=0}^\infty \frac{r_n}{2^n}\phi(2^n x).
\end{equation*}
For example, the {\em alternating Takagi function}
\begin{equation}
\hat{T}(x)=\sum_{n=0}^\infty (-1)^n\frac{\phi(2^n x)}{2^n}
\label{eq:alternating-Takagi}
\end{equation}
is of the above form; see Figure \ref{fig:alternating-Takagi}. This gives an uncountably large class of functions which are ``close" to the Takagi function $T$ in the sense that their partial sums are all piecewise linear with integer slopes that change by $\pm 1$ at each step. It should therefore be no surprise that these functions share a large number of properties with $T$. For instance, Allaart \cite[Section 5]{Allaart4} shows that many of the results from Section \ref{sec:level-sets} concerning level sets hold for arbitrary $F_{\bf r}$: Almost all level sets of $F_{\bf r}$ are finite, but their average cardinality is infinite and the set of ordinates $y$ with uncountably large level sets is residual in the range of $F_{\bf r}$.

\begin{figure}
\begin{center}
\epsfig{file=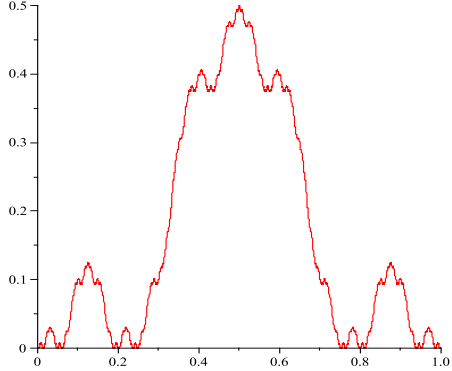, height=.25\textheight, width=.4\textwidth}
\caption{The alternating Takagi function}
\label{fig:alternating-Takagi}
\end{center}
\end{figure}

An interesting random version of the Takagi function is obtained by taking the components of ${\bf r}$ to be independent random variables with $\rP(r_n=1)=p$ and $\rP(r_n=-1)=1-p$, where $0\leq p\leq 1$. The maximum value $M$ of $F_{\bf r}$ is then a random variable, and the set $\MM:=\{x\in[0,1):F(x)=M\}$ is a random set. Allaart \cite{Allaart1} determines the probability distributions of $M$ and the size of $\MM$. If $p<1/2$, the distribution of $M$ is purely atomic and $|\MM|$ is almost surely finite, with range $\{2^l(2^m-1): l\in\ZZ_+, m\in\NN\}$. (For instance, one can have exactly 24 maximum points with positive probability.) If $p\geq 1/2$, the distribution $\mu$ of $M$ is singular continuous, and $\MM$ is a Cantor set with almost-sure Hausdorff dimension $(2p-1)/2p$. In the latter case, Allaart also determines the Hausdorff dimension and multifractal spectrum of $\mu$.

\subsection{The Zygmund spaces $\Lambda_d^*$, $\lambda_d^*$ and $\Lambda_{d,1}^*$} \label{subsec:Zygmund-spaces}

In 1945, Zygmund \cite{Zygmund} introduced the class $\lambda^*$ of ``smooth" functions of period 1, i.e. those functions $f$ satisfying \eqref{eq:Zygmund-smooth} uniformly in $x$, and the wider class $\Lambda^*$ of ``quasismooth" functions of period 1, which satisfy \eqref{eq:Zygmund-smooth} with $O(h)$ replacing $o(h)$. Zygmund studied various properties of functions in these classes, and characterized them in terms of uniform approximation by polynomials.

In 1989, Anderson and Pitt \cite{Anderson} introduced the larger classes $\lambda_d^*$ and $\Lambda_d^*$ of periodic functions satisfying a weaker form of smoothness defined in terms of differences over dyadic intervals. These classes have simple characterizations in terms of the Schauder expansions of their members. Recall that every continuous function $f$ of period 1 which vanishes at $0$ has on $[0,1)$ a unique Schauder expansion of the form \eqref{eq:Schauder-expansion}.
For easier comparison with the Takagi class, we will write the Schauder expansion in the form
\begin{equation}
f(x)=\sum_{n=0}^\infty \frac{r_n(x)}{2^n}\phi(2^n x),
\label{eq:general-Schauder}
\end{equation}
where $r_n(x)$ depends only on the first $n$ binary digits of $x$. More precisely, $r_n(x)=R_n(\eps_1,\dots,\eps_n)$, where $x=\sum_{n=0}^\infty 2^{-n}\eps_n$ and $\eps_n\in\{0,1\}$. The classes $\Lambda_d^*$ and $\lambda_d^*$ are defined as follows: $f\in\Lambda_d^*$ if and only if there exists a uniform bound $M$ such that $|r_n(x)|<M$ for all $n$ and all $x$; and $f\in\lambda_d^*$ if and only if $r_n(x)\to 0$ uniformly in $x$.

Anderson and Pitt \cite{Anderson} show that $\Lambda^*\subset \Lambda_d^*$ and $\lambda^*\subset \lambda_d^*$. The Takagi function is an example of a function which is in $\Lambda_d^*$ but not in $\Lambda^*$. On the other hand, as shown by Abbott et al. \cite{Abbott}, the alternating Takagi function $\hat{T}$ belongs to $\Lambda^*$. In general, functions in $\Lambda^*$ can have corners but no cusps, whereas functions in $\Lambda_d^*$ can have logarithmic cusps. It is shown in \cite{Anderson} that every member $f$ of $\Lambda_d^*$ is Lipschitz of order $h\log(1/h)$. That is, there is a constant $C$ such that, for all $0\leq x<x+h\leq 1$ with $h$ sufficiently small,
\begin{equation*}
|f(x+h)-f(x)|\leq Ch\log(1/h).
\end{equation*}
Another result in \cite{Anderson} is that the graph of every $f\in\Lambda_d^*$ is of $\sigma$-finite linear measure (see Section \ref{sec:graph}). With regard to level sets $L(y)=\{x\in[0,1): f(x)=y\}$, Anderson and Pitt prove that (i) if $f\in\Lambda_d^*$, then $L(y)$ is countable for almost every $y$; and (ii) if $f\in\lambda_d^*$, then $L(y)$ is finite for almost every $y$. It does not appear to be known whether the condition $f\in\Lambda_d^*$ gives enough regularity to the graph of $f$ in order that $L(y)$ be finite for almost all $y$.

The specific case of \eqref{eq:general-Schauder} where $|r_n(x)|$ is constant in $x$ for each $n$ was studied by Allaart \cite{Allaart2}. We shall call the collection of such functions the {\em flexible Takagi class}. It is an immediate generalization of the Takagi class in which the individual ``tents" at each level can point either upward or downward, but all tents within a given level have the same amplitude. This guarantees uniformity in the fine structure across the domain of $f$, while allowing for a wide variety of general shapes of the graph. Indeed, Allaart \cite{Allaart2} manages to extend all of K\^ono's results to this more general setting. Specifically, statements (i)-(iii) of Theorem \ref{thm:Kono-differentiability} hold when $a_n=|r_n|$. Whereas the Takagi class contains (up to a multiplicative constant) only one function in the Zygmund space $\lambda^*$, the flexible Takagi class contains many. For example, it contains the ``bell-shaped" curve
\begin{equation}
f(x)=\begin{cases}
8x^2, & x\leq 1/4\\
8x(1-x)-1, & 1/4\leq x\leq 3/4\\
8(1-x)^2, & x\geq 3/4,
\end{cases}
\label{eq:bell-function}
\end{equation}
depicted in Figure \ref{fig:bell-function} and obtained by setting $r_0=1$, $r_1=0$, and $r_n=-2^{2-n}X_1 X_2$ for $n\geq 2$, where $X_n(x)=(-1)^{\eps_n(x)}$ is the $n$th Rademacher function. See \cite[Section 4]{Allaart2} for more examples. Whether all functions in the flexible Takagi class that belong to $\lambda^*$ must be piecewise quadratic remains open.

\begin{figure}
\begin{center}
\epsfig{file=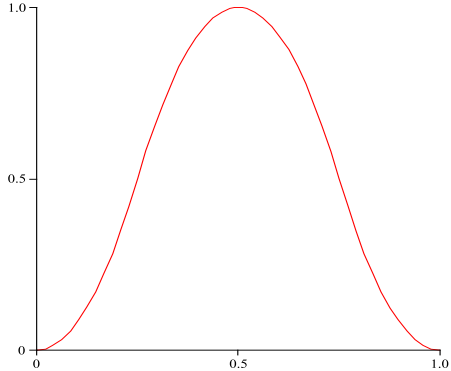, height=.25\textheight, width=.4\textwidth}
\caption{A smooth function in the flexible Takagi class, defined by \eqref{eq:bell-function}.}
\label{fig:bell-function}
\end{center}
\end{figure}

A subcollection of the flexible Takagi class in which $|r_n|=1$ for each $n$ was studied by Abbott, Anderson and Pitt \cite{Abbott}, who denote this subcollection by $\Lambda_{d,1}^*$. It contains the Takagi function, as well as several other interesting functions that have occurred in the literature. For instance, the {\em Gray Takagi function} of Kobayashi \cite{Kobayashi}, which plays a role in the analysis of Gray code digital sums (see Section \ref{subsec:app-number-theory} below), belongs to $\Lambda_{d,1}^*$. It has $r_n=X_n$ for each $n$. Another example is the function $T^3$ of Kawamura \cite{Kawamura1}, which has a connection with certain self-similar sets in the complex plane. It has $r_n=X_1\cdots X_n$ for each $n$. These functions are shown in Figure \ref{fig:Gray-Takagi-and-T3}.

\begin{figure}
\begin{center}
\epsfig{file=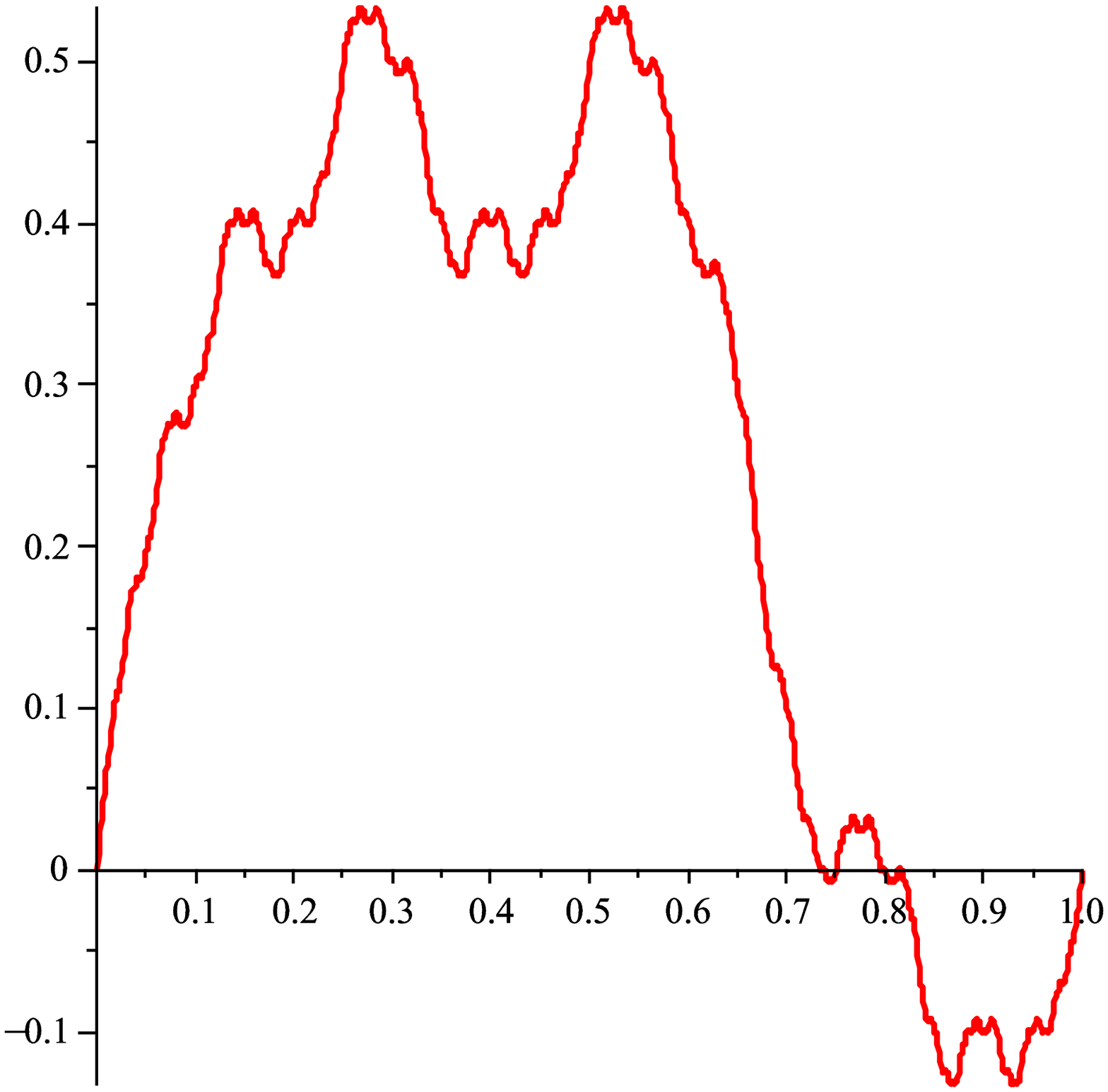, height=.25\textheight, width=.4\textwidth}
\qquad
\epsfig{file=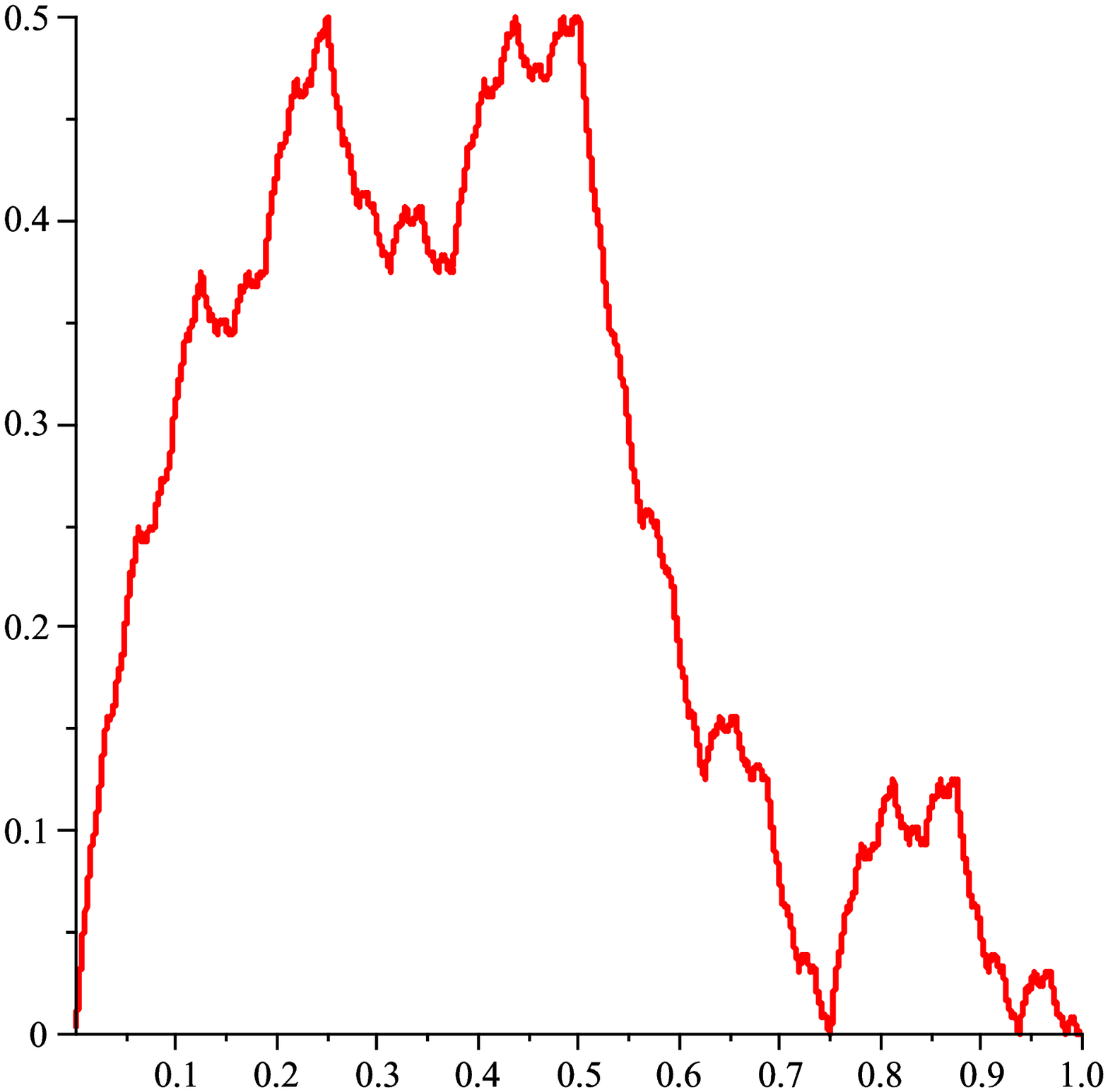, height=.25\textheight, width=.4\textwidth}
\caption{The Gray Takagi function (left) and Kawamura's $T^3$ (right)}
\label{fig:Gray-Takagi-and-T3}
\end{center}
\end{figure}

For a continuous function $f$ on $[0,1)$, define the dyadic difference quotients $D_n f$ as follows. If $x\in[0,1)$, let $k$ be the integer such that $k/2^n\leq x<(k+1)/2^n$, and put
\begin{equation*}
D_n f(x):=\frac{f\left(\frac{k+1}{2^n}\right)-\left(\frac{k}{2^n}\right)}{2^{-n}}=2^n \left[f\left(\frac{k+1}{2^n}\right)-\left(\frac{k}{2^n}\right)\right].
\end{equation*}
(For the Takagi function, $D_n f(x)=D_n(x)$ as in Section \ref{subsec:derivatives}; see \eqref{eq:dyadic-secant-slopes}.) Abbott, Anderson and Pitt \cite{Abbott} consider the set $S_K$ of {\em slow points} with constant $K$, that is, the set
\begin{equation*}
S_K:=S_K(f):=\{x\in[0,1): |D_n f(x)|\leq K\ \mbox{for all $n$}\}.
\end{equation*}
They show that for $f\in\Lambda_{d,1}^*$, the Hausdorff dimension of $S_K$ is given by
\begin{equation*}
\dim_H S_K=1+\log_2\left(\cos\left(\frac{\pi}{2(K+1)}\right)\right).
\end{equation*}
Moreover, their proof shows that $S_K$ is an $s$-set, with $s$ being the above dimension. In particular, the set $\bigcup_{K=1}^\infty S_K$ of all slow points of $f$ has full Hausdorff dimension 1. On the other hand, for $f\in\Lambda_{d,1}^*$ the set of slow points is a null set, because $\{D_n f\}_n$ is a simple random walk and hence obeys the law of the iterated logarithm. The authors of \cite{Abbott} are particularly interested in the interplay between slow points and local Lipschitz properties of a function. Let $L=L(f)$ denote the set of points $x$ at which $f$ satisfies a local Lipschitz condition; that is, those points $x$ for which there exists a constant $M$ such that $|f(x)-f(y)|\leq M|x-y|$ for all $y\neq x$.
Abbott, Anderson and Pitt show that any $f\in\Lambda_{d,1}^*$ satisfies a local Lipschitz condition at most points of $S_K=S_K(f)$ in the sense that $\dim_H (S_K\cap L)=\dim_H S_K$. For 
the Takagi function and the alternating Takagi function the stronger statement $S_K\subset L$ holds, but this is exceptional: If a member $f\in\Lambda_{d,1}^*$ is chosen ``at random", then $\HH^\alpha(S_K\cap L)=0$ with probability 1, where $\alpha=\dim_H S_K$. Since $\HH^\alpha(S_K)>0$, this can be interpreted as saying that the ``typical" $f\in\Lambda_{d,1}^*$ satisfies a Lipschitz condition at very few of its slow points.

\subsection{The functions of Sekiguchi and Shiota} \label{subsec:Sekiguchi-Shiota}

Following the important paper of Hata and Yamaguti \cite{Hata-Yamaguti}, Sekiguchi and Shiota \cite{Sekiguchi-Shiota} studied further generalizions of the Takagi function. Their first result concerns the system of difference equations
\begin{align}
\begin{split}
f\left(\frac{2j+1}{2^{n+1}}\right)-(1-r)f\left(\frac{j}{2^n}\right)&-rf\left(\frac{j+1}{2^n}\right)=c_n,\\
&j=0,1,\dots,2^n-1,\quad n=0,1,2,\dots,
\end{split}
\label{eq:difference-equations}
\end{align}
where $r\in(0,1)$ is a constant parameter.
Let $L_r(x)$ be Lebesgue's singular function, and define the generalized Schauder function $S_r:\RR\to[0,1]$ by
$$S_r(x)=\begin{cases}
L_r(x)/r, &\mbox{if $0\leq x\leq 1/2$}\\
\big(1-L_r(x)\big)/(1-r), &\mbox{if $1/2\leq x\leq 1$},
\end{cases}$$
and $S_r(x+1)=S_r(x)$ for all $x\in\RR$.
Sekiguchi and Shiota show that the system \eqref{eq:difference-equations} has a unique continuous solution $f$ on $[0,1]$ if and only if $\sum_{n=0}^\infty |c_n|<\infty$, in which case
\begin{equation*}
f(x)=f(0)+(f(1)-f(0))L_r(x)+\sum_{n=0}^\infty c_n S_r(2^n x).
\end{equation*}
The second result of \cite{Sekiguchi-Shiota} is that $L_r(x)$ is an analytic function of $r$, and a recursive construction is given in terms of the functions $S_r(x)$ and Rademacher functions of the (normalized) $k$th partial derivative
\begin{equation}
T_{r,k}(x):=\frac{1}{k!}\frac{\partial^k L_r(x)}{\partial r^k}.
\label{eq:Trn-def}
\end{equation}
Sekiguchi and Shiota use martingale theory (based on the connection of $L_r(x)$ with unfair coin tossing explained in Section \ref{sec:FE}) to prove their results. They show moreover that the functions $T_{r,k}$ satisfy the system of difference equations
\begin{align*}
T_{r,k}\left(\frac{2j+1}{2^{n+1}}\right)- (1-r)T_{r,k}\left(\frac{j}{2^n}\right)- &rT_{r,k}\left(\frac{j+1}{2^n}\right) \\
&=T_{r,k-1}\left(\frac{j+1}{2^n}\right)-T_{r,k-1}\left(\frac{j}{2^n}\right)
\end{align*}
for $k\geq 1$, $n\geq 0$ and $j=0,1,\dots,2^n-1$, where we set $T_{r,0}:=L_r$. Note that by \eqref{eq:takagi-lebesgue}, $T_{1/2,1}=2T$.

The special case of the functions $T_{r,k}$ in which $r=1/2$ was investigated further by Allaart and Kawamura \cite{AK1}. 
In that paper we show that $T_{1/2,n}(1-x)=T_{1/2,n}(x)$ when $n$ is odd, and $T_{1/2,n}(1-x)=-T_{1/2,n}(x)$ when $n$ is even. We derive from \eqref{eq:Ulam} and \eqref{eq:Trn-def} the representation
\begin{equation*}
\label{th:Tnx}
T_{1/2,n}(x)=\sum_{k=1}^{\infty}\eps_{k}
\left(\frac12\right)^{k-n}\sum_{i=0}^{n}(-1)^{i}\binom{I_k-1}{i}\binom{k-I_k+1}{n-i}, \qquad 0 \le x \le 1.
\end{equation*}
This is then used to prove that for each $n$, $T_{1/2,n}$ is nowhere differentiable and H\"older continuous of order $h\big(\log(1/h)\big)^n$. That is, there is a constant $C_n$ such that, for $0\leq x<x+h\leq 1$ and $h$ sufficiently small,
\begin{equation*}
|T_{1/2,n}(x+h)-T_{1/2,n}(x)|\leq C_n h\big(\log(1/h)\big)^n,
\end{equation*}
and this bound is the best possible. As a result, the graph of $T_{1/2,n}$ has Hausdorff dimension 1. We also determine the global and local extrema of $T_{1/2,2}$ and $T_{1/2,3}$. The sets of points where these functions attain there absolute maximum are shown to be Cantor sets of Hausdorff dimension zero, and their members have binary expansions that follow a remarkable pattern. By contrast, we conjecture that $T_{1/2,n}$ has only finitely many absolute maximum points when $n\geq 4$. Regarding the growth rate of $M_n:=\max_{0\leq x\leq 1}T_{1/2,n}$, we conjecture that 
\begin{equation}
M_n \sim \frac{2^n}{\sqrt{\pi n}}, \qquad\mbox{as $n\to\infty$}.
\label{eq:maximum-conjecture}
\end{equation}
(Proposition 6.25 in \cite{AK1} shows that $M_n$ grows at least this fast.)

For general $r\in(0,1)$, the functions $T_{r,n}$ have not yet been thoroughly investigated. However, a new paper by de Amo et al. \cite{ADCFS} announces the surprising result that if $r\neq 1/2$, then $T_{r,n}$ is differentiable (with vanishing derivative) almost everywhere for each $n$. This is in sharp contrast with the fact that $T_{r,n}$ is nowhere differentiable if $r=1/2$. One might expect a further study of these functions to reveal many more interesting properties.

\subsection{Other generalizations and variations} \label{subsec:other}

Many other generalizations and variations of the Takagi function have been studied. We briefly mention a few, but the list below is by necessity far from complete. 

Frankl et al. \cite{Frankl} define a generalized Takagi function using arbitrary base $1+c$ instead of base $2$, where $c$ is any positive real number. These functions are used in connection with the Kruskal-Katona theorem in combinatorics; see Section \ref{subsec:app-combinatorics} below. The resulting graphs look like the graph of $T$ with the wind blowing in from the side. Another way to create skewed versions of the Takagi function is of course to take the composition $T\circ h$ of $T$ with an arbitrary homeomorphism $h:[0,1]\to[0,1]$; Kawamura \cite{Kawamura2} considers the case where $h=L_a^{-1}$ and looks at the set of points where $T\circ h$ has a vanishing derivative. Tsujii \cite{Tsujii} constructs a Takagi-like function of two variables. Sumi \cite{Sumi} observes that Lebesgue's singular function $L_a(x)$ can be interpreted as the probability of ``tending to $+\infty$" in a certain kind of random dynamics in the real line. He then extends this notion to random dynamics in the complex plane, obtaining a complex version of Lebesgue's singular function and (by extending the relationship \eqref{eq:takagi-lebesgue}) a complex version of the Takagi function. 

For the more general setting where the tent map $\phi$ is replaced with an arbitrary periodic Lipschitz function, see Mauldin and Williams \cite{Mauldin-Williams} and the many papers citing this work.

\subsection{Open problems} \label{subsec:general-open}

There are many natural questions regarding the functions which have occurred in this section. Some of the problems listed below may have a known answer, some have not yet been seriously investigated, and others may be hard.

\begin{problem}
{\rm
Characterize exactly which functions in the flexible Takagi class belong to Zygmund's space $\lambda^*$. (A partial characterization is given in \cite[Theorem 4.1]{Allaart2}.) Must these functions necessarily be piecewise quadratic?
}
\end{problem}

\begin{problem}
{\rm
The classes $\Lambda^*$ and $\Lambda_{d,1}^*$ are both contained in $\Lambda_d^*$, but they are not disjoint. (For instance, the alternating Takagi function $\hat{T}$ belongs to both.) Can one characterize the intersection $\Lambda^*\cap\Lambda_{d,1}^*$?
}
\end{problem}

\begin{problem}
{\rm
Is it true that for every $f\in\Lambda_{d,1}^*$, the level set $L_f(y):=\{x\in[0,1]: f(x)=y\}$ is finite for almost every $y$?
}
\end{problem}

%\begin{problem}
%{\rm
%For functions $f\in\Lambda_{d,1}^*$, are there simple conditions on the functions $r_n(x)$ in \eqref{eq:general-Schauder} (which must take the values $1$ and $-1$ only) in order that $f$ have a unique maximum and/or minimum point? 
%}
%\end{problem}

\begin{problem}
{\rm
A {\em Besicovitch function} is one that does not possess a one-sided (finite or infinite) derivative at any point. Does the flexible Takagi class contain any Besicovitch functions? 
}
\end{problem}

\begin{problem}
{\rm
At which points do the Takagi-van der Waerden functions $f_r$ possess an infinite derivative? (See \cite{AK2} for the case $r=2$.)
}
\end{problem}

\begin{problem}
{\rm
Determine the level set structure of $f_r$ for $r\geq 3$. In particular, is the level set of $f_r$ at almost every level $y$ finite?
}
\end{problem}

\begin{problem}
{\rm
Prove or disprove that $T_{1/2,n}$ has a unique global maximum point and a unique global minimum point when $n\geq 4$.
}
\end{problem}

\begin{problem}
{\rm
Prove or disprove \eqref{eq:maximum-conjecture}.
}
\end{problem}

\section{Applications} \label{sec:applications}

In this final section, we discuss a number of applications of the Takagi function and its generalizations. We begin with an application to the digital sum problem in number theory.

\subsection{Applications in Number theory} \label{subsec:app-number-theory}

Let $n$ be a positive integer with binary expansion $n=\sum_{i=0}^\infty \alpha_i(n)2^i$, where $\alpha_i(n)\in\{0,1\}$. 
We define the following arithmetical sums:
\begin{align}
s(n)&=\sum_{i=0}^\infty \alpha_i(n), \qquad \mbox{(the binary digital sum)}, \label{eq:binary-digital-sum}\\
S_k(N)&=\sum_{n=0}^{N-1}s(n)^k, \qquad \mbox{(the power sum)}, \label{eq:power-sum}\\
F(\xi,N)&=\sum_{n=0}^{N-1}e^{\xi s(n)}, \qquad \mbox{(the exponential sum)}, \label{eq:exponential-sum}
\end{align}
where $k,N$ are positive integers, and $\xi$ is a real number. Note that $s(n)$ is the number of ones in the binary expansion of $n$. As a special case, $F(\log 2,N)$ is the number of odd numbers in the first $N$ rows of Pascal's triangle.

The power and exponential sums of digital sums were first studied in connection with divisibility problems involving factorials and binomial coefficients, but they occur in other areas of mathematics as well: algebraic number theory, topology, combinatorics and computational algorithms. For more information, see Stolarsky~\cite{Stolarsky}.

If $N$ is a power of $2$, it is clear that 
\begin{equation*}
S_1(N)=\frac{N\log_2 N}{2}, \qquad F(\xi,N)=N^{\log_2 (1+e^{\xi})}.
\end{equation*}
However, it is difficult to get an explicit formula for arbitrary $N \in \NN$. 

The first author to give an exact expression for $S_1(N)$ was Trollope~\cite{Trollope} in 1968. Delange~\cite{Delange} gave a simpler proof and extended the result to digits in arbitrary bases. Trollope's expression is
\begin{equation}
S_1(N)=\frac{N\log_2 N}{2}+E(N),
\label{eq:Trollope}
\end{equation}
where the error term $E(N)$ involves the Takagi function: if $N$ is written as $N=2^m(1+x)$ with $m\in\ZZ_+$ and $0\leq x<1$, then
\begin{equation}
E(N)=2^{m-1}\left\{2x-T(x)-(1+x)\log_2(1+x)\right\}.
\label{eq:error-term}
\end{equation}
We can easily derive this from \eqref{eq:Kruppel} and \eqref{eq:takagi FA}. First, note that \eqref{eq:Kruppel} can be written as
\begin{equation*}
T\left(\frac{k}{2^m}\right)=\frac{mk}{2^m}-\frac{1}{2^{m-1}}S(k).
\end{equation*}
%from which we obtain
%\begin{equation}
%S(k)=\frac{mk}{2}-2^{m-1}T\left(\frac{k}{2^m}\right).
%\end{equation}
So if $N$ is expressed as above, then $N\leq 2^{m+1}$ and
\begin{align*}
S(N)&=\frac{(m+1)N}{2}-2^m T\left(\frac{1+x}{2}\right)\\
&=2^{m-1}\big(m(1+x)+2x-T(x)\big),
\end{align*}
where the last step follows since $T$ satisfies \eqref{eq:takagi FA}. Since
\begin{equation*}
\frac{N\log_2 N}{2}=2^{m-1}(1+x)\big(m+\log_2(1+x)\big),
\end{equation*}
\eqref{eq:error-term} follows.

It is natual to ask if we can generalize the expressions \eqref{eq:Trollope} and \eqref{eq:error-term} to arbitrary $k \in \NN$. For $k=2$, Coquet~\cite{Coquet} obtained an explicit formula and proved that $S_2(N)$ also has a close relationship with a nowhere differentiable function. Unfortunately, the general formula for $S_k(N)$ given by Coquet included a function specified only by a complicated recursion.

It seemed quite difficult to find a direct formula for $S_k(N)$ with $k \geq 3$, but finally, in 1995, Okada et al.~\cite{Okada} gave the complete answer not only for the power sum but also for the exponential sum. The key was to connect $F(\xi, N)$ with Lebesgue's singular function $L_r(x)$. Set $r={(1+e^{\xi})}^{-1}$ and $t=\log_2 N$. Denote the integer part of $t$ by $[t]$, and the fractional part by $\{t\}$. Then 
\begin{equation} \label{eq:exponential-lebesgue}
F(\xi,N)=\frac{1}{r^{[t]+1}}
L_{r}\left(\frac{1}{2^{1-\{t\}}}\right),
\end{equation}
for $\xi \in \RR$ and $N \in \NN$. Using Salem's expression for $L_r(x)$ (See \cite{Salem}), an explicit expression for $F(\xi,N)$ is obtained.

For $S_k(N)$, observe that the following equality holds for $k \in \NN$:
\begin{equation} \label{eq:power-exponential-sum}
S_k(N)=\frac12 \left.\frac{\partial^k}{\partial \xi^k}F(\xi,N)\right|_{\xi=0}.
\end{equation}
An essential part of this derivative is $(\partial^k/\partial r^k)L_r(x)|_{r=1/2}$, and hence, $S_k(N)$ can be expressed explicitly in terms of the nowhere differentiable function $T_{1/2,k}$ of Section \ref{subsec:Sekiguchi-Shiota}.

\bigskip

Later, several analogous problems were studied. 
For instance, instead of the binary expansion of natural numbers, Kobayashi \cite{Kobayashi} considered the {\em Gray code}, which is an encoding of natural numbers as sequences of $0$'s and $1$'s with the property that the representations of adjacent integers differ in exactly one position. Though the Gray code was introduced initially as a solution to a communications problem involving digitization of analogue data, it has since been used in a wide variety of other applications, including databases, experimental design and even puzzle solving; see Savage \cite{Savage}.

Kobayashi defined a probability measure $\tilde{\mu}_r$ on $[0,1]$ analogous to the binomial measure, but whereby each {\em Gray code} digit (rather than binary digit) of $x\in[0,1]$ is $0$ with probability $r$ and $1$ with probability $1-r$, independently of the other digits. 
Adapting the methods of \cite{Sekiguchi-Shiota} and \cite{Okada}, Kobayashi gave explicit expressions for the Gray digital sum, the Gray power sum and the Gray exponential sum, which are defined just as in \eqref{eq:binary-digital-sum}-\eqref{eq:exponential-sum}, but with Gray code digits replacing binary digits. The expressions are in terms of the distribution function $\tilde{L}_r$ of $\tilde{\mu}_r$ and a nowhere-differentiable continuous function $\tilde{T}$ which Kobayashi calls the {\em Gray Takagi function}; see Section \ref{subsec:Zygmund-spaces}.

Kr\"uppel~\cite{Kruppel2} modified the Trollope-Delange formula in a different direction, considering instead of $s(n)$ the alternating binary sum $\hat{s}(n)=\sum_{i=0}^\infty (-1)^i \alpha_i(n)$.
He derived an expression for
$\hat{S}_1(N)=\sum_{n=0}^{N-1}\hat{s}(n)$
in terms of the alternating Takagi function $\hat{T}$ defined in \eqref{eq:alternating-Takagi}.

A comprehensive review of the role of nowhere-differentiable functions and singular measures in the study of digital sum problems is found in Kobayashi et al. \cite{Digital-survey}.

\subsection{Applications in Combinatorics} \label{subsec:app-combinatorics}

A few interesting connections between the Takagi function and the discrete isoperimetric problems in combinatorics have been discovered. 

The first surprising result was given by Frankl et al.~\cite{Frankl} in 1995. 
Let $\binom{\NN}{k}$ denote the collection of subsets of $\NN$ having $k$ elements.
For any family $\FF\subset\binom{\NN}{k}$ and positive integer $l<k$, let the $l$th {\em shadow} of $\FF$ be
$$\Delta_l(\FF):=\{G \in \binom{\NN}{l}:  \exists F \in \FF\ {\mbox s.t.} \ G \subset F\}.$$ 
It is clear that $|\Delta_l(\FF)|$, the size of the shadow, depends on $\FF $. Thus, it is natural to ask for a fixed $m\in\NN$, which family $\FF$ such that 
$|\FF|=m$ attains the minimum size of the shadow. This problem can be viewed as a special case of the {\em vertex discrete isoperimetric problem} by considering a graph whose vertices correspond to finite subsets of $\binom{\NN}{k}$, with an edge connecting a set of cardinality $k$ to each of its subsets of cardinality $k-1$. 

The shadow minimization problem was solved by Kruskal and Katona independently.  
For $A, B \in \binom{\NN}{k}$, define {\em the colex order} by 
$$A <_{colex} B \qquad \Longleftrightarrow \qquad \max\{a \in A\backslash B\} < \max\{b \in B\backslash A\}.$$ 
Denote by $Colex(k,m)$ the family of the first $m$ elements in $\binom{\NN}{k}$ with the colex order. 
The Kruskal-Katona theorem says that for all $\FF \subset \binom{\NN}{k}$ such that $|\FF|=m$, and $l<k$, 
$$\min|\Delta_l(\FF)|=|\Delta_l(Colex(k,m))|.$$ 

Finally, define the {\em Kruskal-Katona function} by 
$$K_l^k(m)=-m+|\Delta_l(Colex(k,m))|.$$ Frankl et al. show that, properly normalized, $K_l^k(m)$ converges uniformly to $T(x)$ if $l=k-1$. 
More precisely, define the shadow function $S_k$ by normalizing $K_l^k$, where $l=k-1$:
\begin{equation*}
S_k(x):=k \binom{2k-1}{k}^{-1} K_{k-1}^{k}\left(\left\lfloor \binom{2k-1}{k} x \right\rfloor\right), \qquad \mbox{for} \qquad 0 \leq x \leq 1.
\end{equation*}
Then $S_k\to T$ uniformly on $[0,1]$.

In general, an estimation of the Kruskal-Katona function is hard to calculate. The result above clearly shows the reason. Frankl et al. exploit the above relationship and various known properties of the Takagi function to give new estimates of the Kruskal-Katona function. They also define a generalized Takagi function (using the base $1+c$ expansion of $x\in[0,1]$ for an arbitrary real $c>0$) and show that it has a similar connection to the minimum shadow size for $l=\lfloor ck \rfloor$.

\bigskip

Another connection of the Takagi function with combinatorics was given by Guu \cite{Guu}. Assume a graph $(V,E)$ with a set of vertices $V$ and a set of edges $E$ is given. For $S\subset V$, let $\theta(S)$ denote the number of edges in $E$ which connect a vertex in $S$ to a vertex in $V\backslash S$. For given $0\leq k\leq |V|$, the {\em edge discrete isoperimetric problem} is to minimize $\theta(S)$ over all $S$ having $k$ elements. Guu considers as a special case the $n$-cube $Q_n=(V_n,E_n)$, and defines the function
$$\theta(n,k)=\min\{\theta(S): S\subset V_n, |S|=k\}.$$
He shows that 
$$\lim_{n\to\infty}\frac{\theta(n,\lfloor 2^n x\rfloor)}{2^n}=T(x).$$

\subsection{Applications in Real Analysis} \label{subsec:app-analysis}

In this subsection we sketch two applications of the Takagi function in real analysis. One concerns the zero sets of continuous nowhere-differentiable functions, while the other relates to the study of approximate convexity of real functions.

\subsubsection{Zero sets of continuous nowhere-differentiable functions}

If a continuous function does not have a finite derivative anywhere, what can one say about its set of zeros? This question was answered in 1966 by Lipinski \cite{Lipinski}, who gave the following characterization.

\begin{theorem}
Let $C\subset [0,1]$. Then $C$ is the zero set of some nonnegative continuous nowhere-differentiable function $f:[0,1]\to\RR$ if and only if $C$ is closed and nowhere dense.
\end{theorem}

It is easy to see that the condition on $C$ is necessary. To prove its sufficiency, Lipinski used the Takagi function to construct an example of a function $f$ with the required properties. To begin, write $[0,1]\backslash C=\bigcup_{n=1}^\infty (a_n,b_n)$ with the union disjoint. (If $[0,1]\backslash C$ is a finite union of open intervals, the construction is easy.) Define functions
\begin{equation*}
T_{a,b}(x)=(b-a)T\left(\frac{x-a}{b-a}\right), \qquad a\leq x\leq b,
\end{equation*}
where $0\leq a<b\leq 1$. Thus, the graph of $T_{a,b}$ is a smaller copy of the graph of $T$ confined to the interval $[a,b]$. It is tempting to construct $f$ by putting
\begin{equation*}
f(x)=\begin{cases}
0, & x\in C\\
T_{a_n,b_n}(x), & x\in(a_n,b_n), n\in\NN.
\end{cases}
\end{equation*}
Indeed, Schubert \cite{Schubert} had mistakenly believed that this creates a nowhere differentiable function. But, as Lipinski points out, $f$ defined this way can be differentiable (with $f'(x)=0)$ at many points of $C$. In fact, if $\lambda(C)>0$, then $f$ is differentiable almost everywhere on $C$. To correct this problem, Lipinski slightly modified the above construction as follows. Enumerate the dyadic open subintervals of $[0,1]$ (i.e. those of the form $(j/2^n,(j+1)/2^n)$) in some arbitrary order as $\{A_n: n\in\NN\}$. For each $n$, let $F_n$ be any interval $(a_{i_n},b_{i_n})$ contained in $A_n$ if such an interval exists, and $F_n=\emptyset$ otherwise, with the restriction that no interval $(a_i,b_i)$ is chosen more than once. Now define a function $\tilde{f}$ by
\begin{equation*}
\tilde{f}(x)=\begin{cases}
0, & x\in C\\
T_{a_i,b_i}(x), & x\in(a_i,b_i), (a_i,b_i)\not\in \{F_n\}\\
\lambda(A_n)T_{a_{i_n},b_{i_n}}(x)/(b_{i_n}-a_{i_n}), & x\in F_n. 
\end{cases}
\end{equation*}
Lipinski shows that $\tilde{f}$ is continuous and nowhere differentiable, and hence has the desired properties. Of course, in this construction we can replace $T$ with any continuous nowhere-differentiable function which is strictly positive in $(0,1)$ and vanishes at $0$ and $1$.

\subsubsection{Approximate convexity}

Our next application is to approximate convexity of real functions. Let $V$ be a convex subset of a normed space and let $\eps\geq 0$, $p>0$ be given constants. Say a function $f: V\to\RR$ is $(\eps,p)$-{\em midconvex} if
$$f\left(\frac{x+y}{2}\right)\leq\frac{f(x)+f(y)}{2}+\eps\|x-y\|^p, \qquad x,y\in V.$$
%and say $f$ is $(\eps,p)$-{\em convex} if
%\begin{equation*}
%f(tx+(1-t)y)\leq tf(x)+(1-t)f(y)+\eps\|x-y\|^p, \qquad x,y\in V, \quad t\in[0,1].
%\end{equation*}
It was shown by H\'azy and P\'ales \cite{Hazy-Pales} that if $f:V\to\RR$ is continuous and $(\eps,1)$-midconvex, then for all $x,y\in V$ and $t\in[0,1]$,
\begin{equation}
f(tx+(1-t)y)\leq tf(x)+(1-t)f(y)+\eps \omega(t)\|x-y\|,
\label{eq:approx-convex}
\end{equation}
where $\omega=2T$. %In particular, $f$ is then $(\delta,1)$-convex with $\delta=4\eps/3$, using the fact from Section 3 that $\max T(x)=2/3$. 
A natural question is whether $\omega$ in this last inequality can be replaced by a smaller function. The negative answer came a few years later, when Boros \cite{Boros} proved that $\omega$ is itself $(1,1)$-convex. In terms of $T$, this comes down to the inequality
\begin{equation*}
T\left(\frac{x+y}{2}\right)\leq \frac{T(x)+T(y)}{2}+\left|\frac{x-y}{2}\right|.
\end{equation*}
(To see that this implies the minimality of $\omega$ in \eqref{eq:approx-convex}, take $f=\eps\omega$, $x=1$ and $y=0$.) Boros' proof was somewhat laborious, with no fewer than eight separate cases in the induction step. But the important thing was that it confirmed P\'ales' conjecture of the minimality of $\omega$.

The next question, then, is what function takes the role of $\omega$ when $p\neq 1$. More precisely, what is the smallest function $\omega_p$ such that, whenever $f:V\to\RR$ is $(\eps,p)$-midconvex, we have
\begin{equation}
f(tx+(1-t)y)\leq tf(x)+(1-t)f(y)+\eps \omega_p(t)\|x-y\|^p
\label{eq:p-convexity}
\end{equation}
for all $x,y\in V$ and $t\in[0,1]$? For $p\in[1,2]$, Tabor and Tabor \cite{Tabor1,Tabor2} showed that the answer is the function
\begin{equation}
\omega_p(x)=2\sum_{n=0}^\infty \frac{1}{2^{np}}\phi(2^n x),
\label{eq:p-Takagi}
\end{equation}
which belongs to the Takagi class. In \cite{Tabor1} they show that any $(\eps,p)$-midconvex function satisfies \eqref{eq:p-convexity}, and in \cite{Tabor2} they establish that $\omega_p$ is itself $(1,p)$-midconvex; that is,
\begin{equation}
\omega_p\left(\frac{x+y}{2}\right)\leq \frac{\omega_p(x)+\omega_p(y)}{2}+|x-y|^p.
\label{eq:approximate-convexity}
\end{equation}
Note that $\omega_2(x)=4x(1-x)$ for $x\in[0,1]$, and $\omega_1(x)=2T(x)$. Tabor and Tabor prove \eqref{eq:approximate-convexity} first for $p=2$, and then deduce it for all $p\in[1,2]$ by expressing $\omega_p$ as an infinite series in terms of $\omega_2$. A different proof of their result (which includes that of Boros) is given by Allaart \cite{Allaart3}, who derives the formula
\begin{equation}
\omega_p\left(\frac{m}{2^n}\right)=\sum_{k=0}^{m-1}\sum_{i=0}^{n-1}\frac{{(-1)}^{\eps_i(k)}}{2^{(n-i-1)p+i}},
\label{eq:expression}
\end{equation}
where $\eps_i(k)\in\{0,1\}$ is determined by $\sum_{i=0}^{n-1}2^i\eps_i(k)=k$. 
Using this expression, \eqref{eq:approximate-convexity} can be reduced to a simple inequality for weighted sums of binary digits, which has an easy induction proof; see \cite{Allaart3}.

Shortly after Tabor and Tabor's work appeared, M\'ako and P\'ales \cite{Mako-Pales} proved a much more general result, which has the following remarkable consequence: For $p\in(0,1]$, the minimal function $\omega_p$ in \eqref{eq:p-convexity} is given by
\begin{equation}
\omega_p(x)=2\sum_{n=0}^\infty \frac{1}{2^n}(\phi(2^n x))^p.
\label{eq:funny-series}
\end{equation}
Comparing \eqref{eq:p-Takagi} and \eqref{eq:funny-series}, it is fascinating to see how at the boundary case $p=1$, the exponent $p$ ``jumps" to the other factor in the series' summand.

For more references to papers concerning approximate convexity, we refer to \cite{Mako-Pales}.

\subsection{Connection with the Riemann Hypothesis} \label{subsec:app-Riemann}

To end this section, we briefly mention a recent result connecting the Takagi function with the Riemann Hypothesis. For a more detailed account, we refer to Lagarias \cite{Lagarias}.

Let $F_n$ be the $n$th Farey sequence (also called Farey series), that is, the set of irreducible fractions in $(0,1]$ with denominator less than or equal to $n$. Then $|F_n|=\Phi(n)$, where $\Phi(n)=\sum_{k\leq n}\varphi(k)$, $\varphi$ being the Euler totient function. The connection between Farey series and the Riemann Hypothesis is well documented and goes back to a 1924 paper by Franel \cite{Franel}.
Recently, Balasubramanian, Kanemitsu and Yoshimoto \cite{Balasub} showed, as an interesting example of their more general theory, that the Riemann Hypothesis is equivalent to the statement
\begin{equation}
\sum_{r\in F_n} T(r)=\frac12\Phi(n)+O(n^{\frac12+\eps}) \qquad\mbox{for every $\eps>0$}.
\label{eq:Takagi-and-Riemann}
\end{equation}
Note that the left hand side involves only function values of rational numbers, and can hence be computed by the method of Section \ref{subsec:rational}. While it may seem unlikely that the Riemann Hypothesis will some day be solved by directly proving \eqref{eq:Takagi-and-Riemann}, the connection is nonetheless surprisingly elegant and beautiful.

\section*{Acknowledgments}
We thank Professors N. Katzourakis and H. Sumi for comments on an earlier draft, and Professor E. de Amo for bringing the paper by Tambs-Lyche to our attention and sending us a preprint of \cite{ADCFS}. We thank Professor J. Lagarias for sending an early draft of his survey \cite{Lagarias}.

\end{document}